\documentclass[11pt]{article}
\usepackage{amsfonts, amsmath, amssymb, latexsym, eucal, amscd} 
\usepackage[dvipdfm]{graphicx}
\usepackage{pstricks,pst-plot,pst-node,pst-text,pst-3d}
\newtheorem{theorem}{Theorem}[section]
\newtheorem{lemma}[theorem]{Lemma}
\newtheorem{proposition}[theorem]{Proposition}
\newtheorem{corollary}[theorem]{Corollary}
\newtheorem{exAux}[theorem]{Example}
\newenvironment{example}{\begin{exAux} \rm}{\end{exAux}}
\newtheorem{Def}[theorem]{Definition}
\newenvironment{definition}{\begin{Def} \rm}{\end{Def}}
\newtheorem{Note}[theorem]{Note}
\newenvironment{note}{\begin{Note} \rm}{\end{Note}}
\newtheorem{Problem}[theorem]{Problem}

\newtheorem{Rem}[theorem]{Remark}

\newtheorem{Not}[theorem]{Notation}

\newtheorem{Conj}[theorem]{Conjecture}

\newtheorem{Ass}[theorem]{Assumption}

\newenvironment{proof}{\medskip\noindent{\bf Proof.\ }}{\qed\medskip}

\newcommand{\qed}{\hfill\mbox{$\Box$\qquad\qquad}}

\newcommand{\F}{\mathbb{F}}

\newcommand{\sltwo}{\mathfrak{sl}_2}
\newcommand{\ad}{\text{\rm ad}\hspace{0.05em}}

\renewcommand{\b}[1]{\langle #1 \rangle}

\newcommand{\bBig}[1]{\Bigl\langle #1 \Bigr\rangle}

\newcommand{\tr}{\text{\rm tr}}
\newcommand{\vphi}{\varphi}
\renewcommand{\th}{\theta}

\newcommand{\twoFone}[4]{\,_2F_1\left(%
\begin{array}{c}#1\;\;#2\\#3\end{array}%
\left|\;#4 \rule{0mm}{3ex}\right.\right)}
\renewcommand{\indent}{\hspace{6mm}}

\begin{document}

\title{Krawtchouk polynomials, the Lie algebra $\mathfrak{sl}_2$,
\\
and Leonard pairs}

\author{Kazumasa Nomura and Paul Terwilliger}

\maketitle

\thispagestyle{empty}

\smallskip

\begin{quote}
\small 
\begin{center}
\bf Abstract
\end{center}
\indent
A Leonard pair is a pair of diagonalizable linear 
transformations of a finite-dimen\-sional vector space,
each of which acts in an irreducible tridiagonal fashion on an 
eigenbasis for the other one.
In the present paper we give an elementary but comprehensive
account of how the following are related:
(i) Krawtchouk polynomials;
(ii) finite-dimensional irreducible modules for the Lie algebra $\sltwo$;
(iii) a class of Leonard pairs said to have Krawtchouk type.
Along the way we obtain elementary proofs of some well-known facts about
Krawtchouk polynomials,
such as the three-term recurrence, the orthogonality,
the difference equation, and the generating function.
The paper is a tutorial meant for a graduate student or a researcher
unfamiliar with the above topics.
\end{quote}

\section{Introduction}

\indent
This paper is about the relationship between the Krawtchouk polynomials,
the Lie algebra $\sltwo$, and a class of Leonard pairs said to have 
Krawtchouk type.
Before going into detail, we take a moment to establish some notation.
Throughout the paper $\F$ denotes a field.
From now until the end of Section \ref{sec:matrices}
we assume the characteristic $\text{Char}(\F) \neq 2$.
Let $N$ denote an integer. 
We now define what it means for $N$ to be {\em feasible}. 
For the case $\text{Char}(\F)=0$, $N$ is feasible whenever $N \geq 0$. 
For the case $\text{Char}(\F)>0$,
$N$ is feasible whenever $0 \leq N < \text{Char}(\F)$.
Let $x$ denote an indeterminate and let $\F[x]$ denote
the $\F$-algebra consisting of the polynomials in $x$ that
have all coefficients in $\F$.
We now define some polynomials in $\F[x]$ called
Krawtchouk polynomials \cite[page 347]{AAR}, \cite[Section 9.11]{KLS}.
Recall the shifted factorial
\begin{align*}
 (\alpha)_n &= \alpha(\alpha+1)\cdots(\alpha+n-1),
    & &  n=0,1,2,\ldots
\end{align*}
We interpret $(\alpha)_0=1$.
By \cite[Section 2.1]{AAR}
the $_2F_1$ hypergeometric series is
\[
 \twoFone{a}{b}{c}{z}
  = \sum_{n=0}^\infty 
     \frac{(a)_n (b)_n}{(c)_n}
     \frac{z^n}{n!}.
\]
The Krawtchouk polynomials are defined using two parameters
denoted $N$ and $p$.
The parameter $N$ is a feasible integer
and the parameter $p$ is a scalar in $\F$ such that 
$p\not=0$ and $p\not=1$. 
For $i=0,1,\ldots,N$  define a polynomial $K_i \in \F[x]$ by
\begin{align}    \label{eq:Kn}
 K_i &= K_i(x;p,N) = 
   \twoFone{-i}{-x}{-N}{\displaystyle \frac{1}{p}}.
\end{align}
We check that $K_i$ is  a well-defined polynomial in $\F[x]$.
Observe that $(-i)_n$ vanishes for $n>i$, 
so in the hypergeometric series \eqref{eq:Kn} the $n$-summand is zero
for $n > i$.
Also observe that $(-N)_n$ is nonzero for $n=0,1,\ldots,i$. 
Therefore the $n$-summand in \eqref{eq:Kn} has nonzero denominator for
$n=0,1,\ldots,i$.
By these comments $K_i$ is a well-defined polynomial in $\F[x]$.
One checks that this polynomial has degree $i$,
and the coefficient of $x^i$ is $\frac{1}{(-N)_i p^i}$.
The polynomial $K_i$ is the $i$th Krawtchouk polynomial
with parameters $N$ and $p$.
By the construction
\begin{align}     \label{eq:AW}
 K_i(j) &= K_j(i),
 & &  i,j=0,1,\ldots,N.
\end{align}
This is an example of a phenomenon known as 
self-duality \cite{Ea} or more generally
Askey-Wilson duality
\cite[Theorem 5.1]{BI}, 
\cite[Theorem 4.1]{T:survey}.

We now recall the notion of a Leonard pair.
A Leonard pair is a pair of diagonalizable linear transformations
of a finite-dimensional vector space, 
each of which acts in an irreducible tridiagonal fashion on an 
eigenbasis for the other one \cite[Definition 1.1]{T:Leonard}.
For instance, for all feasible integers $N$ the pair of matrices
\begin{align*}
  S&= \begin{pmatrix}
       0 & 1 & & & & \text{\bf 0} \\
       N & 0 & 2 \\
         & \cdot & \cdot & \cdot \\
         &   & \cdot & \cdot & \cdot \\
         &   &       & 2     & 0 & N \\
       \text{\bf 0} & & &    & 1 & 0
     \end{pmatrix},
 & D&=\text{diag}(N,N-2,\ldots,-N)
\end{align*}
acts on the vector space $\F^{N+1}$ as a Leonard pair.
To see this see \cite[Section 1]{T:intro}
or Lemma \ref{lem:147} below.
This Leonard pair falls into a family said to have Krawtchouk type
\cite[Example 10.12]{T:survey}.
See Definition \ref{def:177} below for the definition of Krawtchouk type.
See \cite[Example 1.5]{ITT}, \cite[Example 1.3]{TV} for more examples
of Leonard pairs that have Krawtchouk type.

In the present paper
we give an elementary but comprehensive account of how the following
are related:
(i) Krawtchouk polynomials;
(ii) finite-dimensional irreducible $\sltwo$-modules;
(iii) Leonard pairs of Krawtchouk type.
The paper is a tutorial meant for a graduate student or
a researcher unfamiliar with the above topics.
In this regard the paper is similar to a paper of Junie Go \cite{Go}
which provides an introduction to the subconstituent algebra 
\cite{T:subconst1} using the hypercube as a concrete example.

Before summarizing the present paper we briefly review the history
concerning how the Krawtchouk polynomials are related to $\sltwo$.
A relationship between the Krawtchouk polynomials and $\sltwo$
was first given by Miller \cite{Mil};
he observed that the difference equations for Krawtchouk polynomials 
come from the irreducible representations of $\sltwo$.
Koornwinder \cite[Section 2]{Koo} observed that the matrix elements
of a finite-dimensional irreducible representation of the
group $SU(2)$ can be written in terms of Krawtchouk polynomials.
This gives a connection between the Krawtchouk polynomials and
$\sltwo$ since the irreducible representations of $SU(2)$ and
$\sltwo$ are essentially the same.
See \cite[Section 2]{Koe}, \cite[Sections 1,2]{Ros2}, 
\cite[Section 6.8.1]{ViK} for more work on this topic.
Later there appeared some articles that gave a connection
between Krawtchouk polynomials and $\sltwo$:
\cite{Fe1}, \cite{Fe}, \cite[Section 4]{FeK}, \cite[Chapter 5, IV]{FeS}.
In each of these articles, the above pair $S,D$ 
acts as a bridge between $\sltwo$ and Krawtchouk polynomials.
On one hand, the matrix $S$ (resp. $D$) represents the action of
$e+f$ (resp. $h$) on the irreducible $\sltwo$-module with
dimension $N+1$.
Here $e,f,h$ denote the usual Chevalley basis for $\sltwo$.
On the other hand $S$ and $D$ are related to the Krawtchouk polynomials
$K_i(x;1/2,N)$ in the following way.
Sylvester \cite{Syl} observed that the matrix $S$ has eigenvalues
$\{N-2i\}_{i=0}^N$;
this was recalled by Askey in \cite[Section 1]{As}.
Since $S$ has mutually distinct eigenvalues $\{N-2i\}_{i=0}^N$,
there exists an invertible matrix $P$ such that $PSP^{-1} = D$.
It turns out that, after a suitable normalization,
the entries of $P$ are
\begin{align*}
   P_{ij} &= \binom{N}{i} K_i(j;1/2,N),
   & i,j&=0,1,\ldots,N.
\end{align*}
As far as we know this fact was first observed by Kac \cite[Section 4]{Ka}
in the context of probability theory.
It later appeared in combinatorics, in the context of
the Hamming association scheme \cite[Theorem 4.2]{De}; see also
\cite[Theorem 6]{Sl} and \cite[III.2]{BI}.

We now summarize the contents of the present paper.
We consider a type of element in $\sltwo$ said to be normalized semisimple.
Our main object of study is a pair $a,a^*$ of normalized 
semisimple elements that generate $\sltwo$.
We show that $a,a^*$ satisfy a pair of relations
\begin{align*}
  [a,[a,a^*]] &= 4(2p-1) a + 4a^*,  \\
  [a^*,[a^*,a]] &= 4(2p-1)a^* + 4a,
\end{align*}
where the scalar $p$ depends on the $\sltwo$ Killing form applied to $a,a^*$.
The above equations are a special case of the Askey-Wilson relations 
\cite[(3.2)]{GLZ}, \cite[Theorem 1.5]{TV}. 
We show that $\sltwo$ has a presentation involving generators $a,a^*$
subject to the above relations.
We describe $\sltwo$ from the point of view of this presentation.
We show that $\sltwo$ admits an antiautomorphism $\dagger$ that 
fixes each of $a,a^*$.
For all feasible integers $N$
we consider an $(N+1)$-dimensional irreducible $\sltwo$-module $V$ 
consisting of the homogeneous polynomials in two variables that have 
total degree $N$.
We display a nondegenerate symmetric bilinear form
$\b{\;,\;}$ on $V$ such that
$\b{\vphi.u,v}=\b{u,\vphi^\dagger.v}$ for all
$\vphi \in \sltwo$ and $u,v \in V$.
We display two bases for $V$, denoted
$\{v_i\}_{i=0}^N$ and $\{v^*_i\}_{i=0}^N$;
the basis $\{v_i\}_{i=0}^N$ diagonalizes $a$ and the basis
$\{v^*_i\}_{i=0}^N$ diagonalizes $a^*$.
We show that each of these bases is orthogonal with respect to
$\b{\;,\;}$.
We show that
\begin{align*}
  \b{v_i,v^*_j} &= K_i(j;p,N),
 & i,j &=0,1,\ldots,N.
\end{align*}
Using these results we recover some well-known facts about
Krawtchouk polynomials,
such as the three-term recurrence, the orthogonality,
the difference equation, and the generating function.
We interpret these facts in terms of matrices.
Finally we show that the pair $a,a^*$ acts on the above $\sltwo$-module $V$ 
as a Leonard pair of Krawtchouk type, and every Leonard pair of Krawtchouk 
type is obtained in this way.

The paper is organized as follows.
In Section \ref{sec:sl2}, after recalling some basic materials
concerning $\sltwo$, 
we describe a pair of normalized semisimple elements that generate $\sltwo$.
In Section \ref{sec:sl2Krawt} we describe how finite-dimensional
irreducible $\sltwo$-modules look from the point of view of these elements.
In this description we make heavy use of Krawtchouk polynomials.
Along the way we recover some well-known facts about Krawtchouk polynomials.
In Section \ref{sec:matrices} 
these facts are interpreted in terms of matrices.
In Section \ref{sec:LP} we bring in the notion of a Leonard pair. 
After obtaining some basic facts about general Leonard pairs,
we focus on Leonard pairs of Krawtchouk type. 
In Section \ref{sec:LPKrawt} we characterize Leonard pairs
of Krawtchouk type as described in the last sentence of the previous 
paragraph.

\section{The Lie algebra $\mathfrak{sl}_2(\F)$}
\label{sec:sl2}

\indent
Throughout this section assume $\F$ is algebraically closed.
For all integers $n \geq 1$ let $\text{Mat}_n(\F)$ denote the $\F$-algebra 
consisting of the
$n\times n$ matrices that have all entries in $\F$.

\medskip

The Lie algebra $\sltwo(\F)$ consists of the matrices in $\text{Mat}_2(\F)$
that have trace $0$, together with the Lie bracket $[y,z]=yz-zy$.
We abbreviate $L=\sltwo(\F)$.
$L$ has a basis
\begin{align}          \label{eq:ehf}
 e &= \begin{pmatrix}
      0 & 1 \\
      0 & 0
    \end{pmatrix},
&
 h &= \begin{pmatrix}
       1 & 0 \\
       0 & -1
      \end{pmatrix},
&
 f &= \begin{pmatrix}
       0 & 0 \\
       1 & 0
      \end{pmatrix}.
\end{align}
This basis satisfies
\begin{align}     \label{eq:refehf}
 [h,e] &= 2e,  &  [h,f] &= -2f, & [e,f] &= h.
\end{align}

\medskip

For $y \in L$ the {\em adjoint map}
$\ad y : L \to L$ is the $\F$-linear transformation that sends 
$z \mapsto [y,z]$ for $z \in L$.
The {\em Killing form} is the bilinear form
$(\;,\;) :  L \times L \to  \F$ such that 
$(y,z) = \tr(\ad y \> \ad z)$ for $y,z \in L$, where $\tr$ means trace.
For notational convenience define a bilinear form
$\b{\;,\;} : L \times L \to \F$ such that
$\b{y,z} = \frac{1}{8} (y,z)$ for $y,z \in L$.
We abbreviate $\|y\|^2 = \b{y,y}$.
The values of $\b{\;,\;}$ 
on the elements \eqref{eq:ehf} are given as follows.
\begin{equation}     \label{eq:bilin}
 \begin{array}{c|ccc}
  \b{\;,\;}  & e & h & f \\   \hline
  e & 0 & 0 & \frac{1}{2} \\
  h & 0 & 1 & 0 \\
  f & \frac{1}{2} & 0 & 0
 \end{array}
\end{equation}
Using \eqref{eq:bilin} one finds $\b{y,z} = \tr(yz)/2$ 
for $y,z \in L$.

Pick $y \in L$ and write
\begin{equation}     \label{eq:y}
  y  = \begin{pmatrix}
        \alpha & \beta \\
        \gamma & - \alpha
      \end{pmatrix}
    = \beta e + \alpha h + \gamma f.
\end{equation}
Then $\|y\|^2 = \alpha^2 + \beta\gamma = -\det(y)$.
Let $r,s$ denote the eigenvalues of $y$.
Then $r+s=0$ and $rs=-\|y\|^2$.

\medskip

By an {\em automorphism} of $L$ we mean an isomorphism of $\F$-vector spaces
$\sigma : L \to L$ such that $[y,z]^\sigma=[y^\sigma,z^\sigma]$ for $y,z \in L$.
Let $\sigma$ denote an automorphism of $L$.
Observe that $\ad (y^\sigma) = \sigma(\ad y)\sigma^{-1}$ for $y \in L$.
Using this we find $\b{y,z} = \b{y^\sigma,z^\sigma}$ for $y,z \in L$.
In particular $\|y\|^2=\|y^\sigma\|^2$ for $y \in L$.

\medskip

The following result is well-known \cite[Section 2.3]{Humph};
we give a short proof for the sake of completeness.

\medskip

\begin{lemma}  {\rm \cite[Section 2.3]{Humph}}    \label{lem:autoinner}   \samepage
The following hold.
\begin{itemize}
\item[\rm (i)]
Let $M$ denote an invertible matrix in $\text{\rm Mat}_2(\F)$.
Then the map $L \to L$, $y \mapsto MyM^{-1}$ is an automorphism of $L$.
\item[\rm (ii)]
Let $\sigma$ denote an automorphism of $L$.
Then there exists an invertible $M \in \text{\rm Mat}_2(\F)$
such that $y^\sigma = MyM^{-1}$ for $y \in L$.
\end{itemize}
\end{lemma}

\begin{proof}
(i): Clear.

(ii):
By \eqref{eq:ehf} $h$ has eigenvalues $1$, $-1$.
Observe that $\| h^\sigma \|^2=\|h\|^2$ so $h^\sigma$ has eigenvalues
$1$, $-1$.
Therefore there exists an invertible $P \in \text{Mat}_2(\F)$ such that 
$h^\sigma = PhP^{-1}$.
Since the map $L \to L$, $y \mapsto PyP^{-1}$ is an automorphism of $L$, 
we may assume without loss of generality that $h^\sigma=h$.
By \eqref{eq:refehf} the element $e$ is a basis for the eigenspace of 
$\ad h$ associated with the eigenvalue $2$.
Applying $\sigma$ and using $h^\sigma=h$ we see that 
$e^\sigma$ is in this eigenspace. 
Therefore there exists $\beta \in \F$ such that $e^\sigma = \beta e$. 
Similarly there exists $\gamma \in \F$ such that 
$f^\sigma = \gamma f$. 
Using $[e^\sigma,f^\sigma] = h$ we find $\beta \gamma =1$.
Define the matrix $M=  \text{diag}(\beta, 1)$.
By the above comments $y^\sigma = MyM^{-1}$ for all $y \in \{e,h,f\}$. 
Therefore $y^\sigma = MyM^{-1}$ for all $y \in L$.
\end{proof}

\medskip

We recall a few definitions. 
Let $V$ denote a nonzero  finite-dimensional vector space over $\F$ and let 
$A:V \to V$ denote an $\F$-linear transformation.
We say that $A$ is {\it diagonalizable} whenever
$V$ has a basis consisting of eigenvectors for $A$.
Let $\{u_i\}_{i=1}^n$ denote a basis for $V$. 
For $B \in \text{Mat}_n(\F)$ we say that 
{\em $B$ represents $A$ with respect to $\{u_i\}_{i=1}^n$}
whenever $Au_j=\sum_{i=1}^n B_{ij}u_i$ for $j=1,2,\ldots,n$.

\medskip

\begin{lemma}      \label{lem:matrixady}    \samepage
Let $y \in L$ be as in \eqref{eq:y}.
Then with respect to the basis $e,h,f$ the matrix representing 
the $\F$-linear transformation $\ad y : L \to L$ is
\begin{equation}    \label{eq:matady}
 \begin{pmatrix}
   2 \alpha & -2 \beta & 0 \\
  -\gamma & 0 & \beta  \\
  0 & 2\gamma & -2\alpha
 \end{pmatrix}.
\end{equation}
\end{lemma}

\begin{proof}
Use \eqref{eq:refehf}.
\end{proof}

\begin{corollary}     \label{cor:ady}   \samepage
Let $y$ denote an element of $L$ and let $r$, $-r$ denote the
eigenvalues of $y$.
Then the eigenvalues of the $\F$-linear transformation
$\ad y : L \to L$ are $2r,0,-2r$.
\end{corollary}

\begin{proof}
Without loss we may assume that $y$ is from \eqref{eq:y},
so that Lemma \ref{lem:matrixady} applies.
Compute the characteristic polynomial of \eqref{eq:matady}
and simplify using $r^2=\alpha^2+\beta\gamma$.
\end{proof}

 \medskip

An element $y \in L$ is said to be {\em semisimple} whenever 
the $\F$-linear transformation $\ad y : L \to L$ is diagonalizable.
Let $\sigma$ denote an automorphism of $L$.
Then $y$ is semisimple if and only if $y^\sigma$ is semisimple.

\medskip

Let $0 \neq y \in L$ and let $r$, $-r$ denote the eigenvalues of $y$.
We have two cases:

\medskip

\begin{itemize}
\item[] Case $r=0$:
 $\;\;\;\;  y^2=0$, $\;\; \|y\|^2=0$, $\;\; \text{det}(y)=0$;
\item[] Case $r \neq 0$:
 $\;\;\;\;  y$ is diagonalizable,
 $\;\; \|y\|^2 \neq 0$, $\;\; \text{det}(y) \neq 0$.
\end{itemize} 

\medskip

The following result is well-known \cite[Section 4.2]{Humph};
we give a short proof for the sake of completeness.

\medskip

\begin{lemma}  {\rm  \cite[Section 4.2]{Humph} } \label{lem:ss}    \samepage
For $y \in L$ the following are equivalent:
\begin{itemize}
\item[\rm (i)]
$y$ is semisimple.
\item[\rm (ii)]
$y$ is diagonalizable.
\end{itemize}
\end{lemma}

\begin{proof}
To avoid trivialities assume $y \not=0$.
Let $r$, $-r$ denote the eigenvalues of $y$. 
First assume that $y$ is diagonalizable. 
We have $r\not= 0$ and  $\text{Char}(\F) \not=2$ so 
$2r$, $0$, $-2r$ are mutually distinct. 
Now $\ad y$ is diagonalizable so $y$ is semisimple.
Next assume that $y$ is not diagonalizable. 
Then $r=0$ so $\ad y$ has all eigenvalues zero. 
The linear transformation $\ad y$ is nonzero and nilpotent.
Therefore $\ad y$ is not diagonalizable so $y$ is not semisimple.
\end{proof}

\medskip

For the following lemma the proof is routine and left to the reader.

\medskip

\begin{lemma}                   \label{lem:2.1}   \samepage
For $y \in L$ the following are equivalent:
\begin{itemize}
\item[\rm (i)]
$\|y\|^2=1$.
\item[\rm (ii)]
$\det(y)=-1$.
\item[\rm (iii)]
$y$ is diagonalizable with eigenvalues $1$, $-1$.
\item[\rm (iv)]
There exists an automorphism of $L$ that sends $y \mapsto h$.
\end{itemize}
\end{lemma}

\medskip

Given a semisimple $y \in L$,
we say $y$ is {\em normalized} whenever $\|y\|^2=1$.

\medskip

\begin{definition}   \label{def:p} \samepage
For a pair $a,a^*$ of normalized semisimple elements of $L$,
we define $p \in \F$ such that $\b{a,a^*}=1-2p$.
We call $p$ the {\em corresponding parameter} for the pair $a,a^*$.
\end{definition}

\begin{example}     \label{exam:123}
Consider the pair of matrices
\begin{align}               \label{eq:exam123}
 a &= \begin{pmatrix}
        \alpha & \beta \\
        \gamma & - \alpha
     \end{pmatrix},
 & 
 a^* &= \begin{pmatrix}
          1 & 0 \\
          0 & -1
        \end{pmatrix},
\end{align}
where $\alpha^2+\beta\gamma=1$.
Observe that each of $a,a^*$ is a normalized semisimple element of $L$.
For this pair the corresponding parameter $p$ satisfies
$p=(1-\alpha)/2$ since $\b{a,a^*}=\alpha$.
Note that $\alpha=1-2p$ and $\beta\gamma=4p(1-p)$.
\end{example}

\begin{example}    \label{exam:123b}
Consider the pair of matrices
\begin{align}           \label{eq:exam123b}    \samepage
 a &= \begin{pmatrix}
        1-2p & 2(1-p)  \\
        2p   & 2p-1
       \end{pmatrix},
&
 a^* &= \begin{pmatrix}
         1 & 0 \\
         0 & -1
        \end{pmatrix},
\end{align}
where $p \in \F$.
Then each of $a,a^*$ is a normalized semisimple element of $L$,
and $p$ is the corresponding parameter.
This is a special case of Example \ref{exam:123} with
$\beta=2(1-p)$ and $\gamma=2p$.
\end{example}

\begin{lemma}     \label{lem:124}     \samepage
Let $a,a^*$ denote a pair of normalized semisimple elements of $L$,
and let $p$ denote the corresponding parameter.
Then the values of $\b{\;,\;}$ on the elements $a$, $a^*$, $[a,a^*]$
are given as follows:
\[
 \begin{array}{c|ccc}
  \b{\;,\;} & a & a^* & [a,a^*]  \\ \hline
  a   &  1  & 1-2p & 0 
\\
  a^* & 1-2p & 1 & 0 
\\ \,
  [a,a^*] \, & 0 & 0 & -16p(1-p)
 \end{array}
\]
The above matrix has determinant $-64p^2(1-p)^2$.
\end{lemma}

\begin{proof}
By Lemma \ref{lem:2.1}(iv) we may assume 
without loss of generality 
that $a,a^*$ are from Example \ref{exam:123}.
Observe that $[a,a^*]=-2\beta e+2\gamma f$.
Using this and \eqref{eq:bilin} we routinely obtain the results.
\end{proof}

\begin{lemma}     \label{lem:125}     \samepage
Let $a,a^*$ denote a pair of normalized semisimple elements of $L$,
and let $p$ denote the corresponding parameter.
Then the following are equivalent:
\begin{itemize}
\item[\rm (i)]
$p \neq 0$ and $p \neq 1$.
\item[\rm (ii)]
The elements $a$, $a^*$, $[a,a^*]$ form a basis for $L$.
\item[\rm (iii)]
The elements $a$, $a^*$ generate $L$.
\end{itemize}
\end{lemma}

\begin{proof}
(i)$\Rightarrow$(ii):
In Lemma \ref{lem:124} the matrix of inner products is nonsingular.
Therefore the elements $a$, $a^*$, $[a,a^*]$ are linearly independent
and hence form a basis for $L$.

(ii)$\Rightarrow$(i):
The matrix \eqref{eq:bilin} is nonsingular so the bilinear form $\b{\;,\;}$ 
is nondegenerate on $L$. 
Therefore the matrix in Lemma \ref{lem:124} has nonzero determinant.

(ii)$\Rightarrow$(iii):
Clear.

(iii)$\Rightarrow$(ii):
Observe $[a,a^*] \not\in \text{Span}\{a,a^*\}$;
otherwise $\text{Span}\{a,a^*\}$ is a Lie subalgebra of $L$ 
which contradicts our assumption that $a,a^*$ generate $L$.
Note that $a,a^*$ are linearly independent;
otherwise $[a,a^*]=0$.
By these comments $a,a^*,[a,a^*]$ are linearly independent and
hence form a basis for $L$.
\end{proof}

\begin{lemma}   \label{lem:aasbbs}    \samepage
Let $a,a^*$ and $b,b^*$ denote pairs of normalized semisimple elements
of $L$, each of which generate $L$.
Then the following are equivalent:
\begin{itemize}
\item[\rm (i)]
$\b{a,a^*} = \b{b,b^*}$.
\item[\rm (ii)]
There exists an automorphism of $L$ that sends
$a \mapsto b$ and $a^* \mapsto b^*$.
\end{itemize}
\end{lemma}

\begin{proof}
(i)$\Rightarrow$(ii):
The pairs $a,a^*$ and $b,b^*$ have the same corresponding parameter,
which we denote by $p$.
Note that $p \neq 0$, $p \neq 1$ by Lemma \ref{lem:125}.
We first show that there exists an automorphism of $L$ that sends
$a,a^*$ to the pair \eqref{eq:exam123b}.
By Lemma \ref{lem:2.1}(iv) we may assume that 
$a,a^*$ are as in \eqref{eq:exam123}.
Note that $\gamma \neq 0$ since $\beta\gamma = 4p(1-p)$ and
$p \neq 0$, $p \neq 1$.
Define $M=\text{diag}(1,2p\gamma^{-1})$.
Then $M$ is invertible, and the automorphism $L \to L$, $y \mapsto MyM^{-1}$
sends $a,a^*$ to the pair \eqref{eq:exam123b}.
Similarly there exists an automorphism of $L$ that sends $b,b^*$ to
the pair \eqref{eq:exam123b}.
The result follows.

(ii)$\Rightarrow$(i):
Clear.
\end{proof}

\begin{lemma}   \label{lem:126}     \samepage
Let $a,a^*$ denote a pair of normalized semisimple elements of $L$,
and let $p$ denote the corresponding parameter.
Then
\begin{align}
 [a,[a,a^*]] &= 4(2p-1)a+4a^*,    \label{eq:lem124a}  \\
 [a^*,[a^*,a]] &= 4(2p-1)a^* + 4a.  \label{eq:lem124b}
\end{align}
\end{lemma}

\begin{proof}
By Lemma \ref{lem:aasbbs}  we may assume that $a,a^*$ are as in
\eqref{eq:exam123b}.
Using the matrices in \eqref{eq:exam123b} we routinely
verify  \eqref{eq:lem124a}, \eqref{eq:lem124b}.
\end{proof}

\begin{lemma}    \label{lem:127}     \samepage
Fix $p \in \F$ such that $p \neq 0$, $p \neq 1$.
Let $\cal L$ denote the Lie algebra over $\F$ defined by
generators $u,v$ and relations
\begin{align}
 [u,[u,v]] &= 4(2p-1)u + 4v,     \label{eq:thm127a}\\
 [v,[v,u]] &= 4(2p-1)v + 4u.     \label{eq:thm127b}
\end{align}
Then $\cal L$ is isomorphic to $L$.
Moreover each of $u,v$ is normalized semisimple,
and $p$ is the corresponding parameter.
\end{lemma}

\begin{proof}
Let the pair $a,a^*$ be from Example \ref{exam:123b}.
Note that $a,a^*$ is a pair of normalized semisimple elements
of $L$, and $p$ is the corresponding parameter.
We display a Lie algebra isomorphism ${\cal L} \to L$
that sends $u \mapsto a$ and $v \mapsto a^*$.
By Lemma \ref{lem:126} the elements $a,a^*$ satisfy
\eqref{eq:lem124a}, \eqref{eq:lem124b}.
Comparing these relations with \eqref{eq:thm127a}, \eqref{eq:thm127b}
we see that there exists a Lie algebra homomorphism
$\vphi: {\cal L} \to L$ that sends $u \mapsto a$ and $v \mapsto a^*$.
We show this homomorphism is bijective.
By Lemma \ref{lem:125} the elements $a,a^*$ generate $L$,
so $\vphi$ is surjective.
Therefore $\dim {\cal L} \geq 3$.
Using \eqref{eq:thm127a} and \eqref{eq:thm127b}
we find $\cal L$ is spanned by $u,v,[u,v]$.
Therefore $\dim {\cal L} \leq 3$.
By these comments $\dim {\cal L}=3$ and $\vphi$ is bijective.
We have shown $\vphi$ is an isomorphism of Lie algebras.
The result follows.
\end{proof}

\medskip

Let $a,a^*$ denote normalized semisimple elements that generate $L$,
and let $p$ denote the corresponding parameter.
By Lemma \ref{lem:125} $p \neq 0$ and $p \neq 1$.
By Lemma \ref{lem:aasbbs} there exists an automorphism of $L$ that sends
$a,a^*$ to the pair \eqref{eq:exam123b}.
So without loss of generality we may assume $a,a^*$ is the pair 
\eqref{eq:exam123b} with $p \neq 0$, $p \neq 1$.
This assumption will be in effect until the end of Section \ref{sec:matrices}.
Thus
\begin{align}     \label{eq:aas}
  a &= 2(1-p)e+(1-2p)h+2pf,  &  a^* &= h.
\end{align}
Observe
\begin{equation}   \label{eq:aas2}
 [a,a^*] = 4(p-1)e + 4pf.
\end{equation}
By Lemma \ref{lem:125} the elements $a,a^*,[a,a^*]$ form a basis for $L$.

\medskip

\begin{lemma}   \label{lem:128}     \samepage
There exists a unique automorphism of $L$ that sends
$a \mapsto a^*$ and $a^* \mapsto a$.
Denoting this automorphism by $*$ we have
$(y^*)^*=y$ for $y \in L$.
\end{lemma}

\begin{proof}
In Lemma \ref{lem:127} the relations are invariant under the map $u \mapsto v$,
$v \mapsto u$.
Therefore the automorphism exists.
This automorphism is unique since $a,a^*$ generate $L$.
The last assertion is clear.
\end{proof}

\medskip

Let $U,W$ denote the following matrices in $\text{Mat}_2(\F)$:
\begin{align}         \label{eq:defUW}
U &= \begin{pmatrix}
       1 & 1 \\
       1 & 1-p^{-1}
     \end{pmatrix},
&
W &= \begin{pmatrix}
      1-p & 0 \\
      0 & p
     \end{pmatrix}.
\end{align}
One checks $WUWU = (1-p)I$.
Define $R = WU$, so that $R^2 = (1-p)I$.
We have
\begin{align}               \label{eq:R}
 R &= \begin{pmatrix}
       1-p & 1-p \\
       p & p-1
     \end{pmatrix},
 &
 R^{-1} &= 
      \begin{pmatrix}
        1 & 1 \\
        \frac{p}{1-p} & -1
      \end{pmatrix}.
\end{align}

\medskip

\begin{lemma}    \label{lem:129}     \samepage
For $y \in L$ we have $y^* = RyR^{-1}$.
\end{lemma}

\begin{proof}
Observe that the map $L \to L$, $y \mapsto RyR^{-1}$ is an automorphism
of $L$.
Using \eqref{eq:exam123b} and \eqref{eq:R}
one checks $RaR^{-1} = a^*$ and $Ra^*R^{-1}=a$.
The result follows in view of Lemma \ref{lem:128}.
\end{proof}

\medskip

Recall that $e,h,f$ is a basis for $L$.
Applying the map $*$ to this basis
we get another basis $e^*$, $h^*$, $f^*$ for $L$.
By \eqref{eq:ehf}, \eqref{eq:refehf}, \eqref{eq:R} and Lemma \ref{lem:129},
\begin{align}
 e^* &= (p-1)e + ph + \frac{p^2}{1-p} f,    \label{eq:es} 
\\
 h^* &= 2(1-p)e + (1-2p)h + 2pf = a,        \label{eq:hs}
\\
 f^* &= (1-p)e + (1-p)h + (p-1)f.           \label{eq:fs}
\end{align}

\medskip

In summary we have the following three bases for $L$:
\begin{align}    \label{eq:threebases}
 & e,h,f; &  & a,a^*,[a,a^*]; &    & e^*,h^*,f^*.
\end{align}

\medskip

We recall the notion of a transition matrix.
Let $V$ denote a nonzero finite-dimensional vector space over $\F$
and let $\{u_i\}_{i=1}^n$, $\{v_i\}_{i=1}^n$ denote two bases for $V$.
By the {\em transition matrix} from $\{u_i\}_{i=1}^n$ to
$\{v_i\}_{i=1}^n$ we mean the matrix $T \in \text{Mat}_n(\F)$
such that $v_j = \sum_{i=1}^n T_{ij}u_i$ for $j=1,2,\ldots,n$.
In this case $T^{-1}$ exists, and equals the transition matrix
from $\{v_i\}_{i=1}^n$ to $\{u_i\}_{i=1}^n$.
Let $A : V \to V$ denote an $\F$-linear transformation
and let $B \in \text{Mat}_n(\F)$ denote the matrix that represents $A$
with respect to $\{u_i\}_{i=1}^n$.
Then the matrix $T^{-1}BT$ represents $A$ with respect to $\{v_i\}_{i=1}^n$.
Let $\{w_i\}_{i=1}^n$ denote a basis for $V$ and let
$S$ denote the transition matrix from $\{v_i\}_{i=1}^n$ to
$\{w_i\}_{i=1}^n$.
Then $TS$ is the transition matrix from $\{u_i\}_{i=1}^n$
to $\{w_i\}_{i=1}^n$.

\medskip

\begin{lemma}   \label{lem:132}
For the bases \eqref{eq:threebases} the transition matrices are given as follows:
\begin{itemize}
\item[\rm (i)]
The transition matrix from the basis $e,h,f$ to the basis $a,a^*,[a,a^*]$ is
\[
 \begin{pmatrix}
  2(1-p) & 0 & 4(p-1) \\
  1-2p & 1 & 0 \\
  2p & 0 & 4p
 \end{pmatrix},
\]
and the transition matrix from the basis $a,a^*,[a,a^*]$ to the basis $e,h,f$ is
\[
 \begin{pmatrix}
  \frac{1}{4(1-p)}  & 0 & \frac{1}{4p}   \\
  \frac{2p-1}{4(1-p)}  & 1 &  \frac{2p-1}{4p}  \\
  \frac{1}{8(p-1)}  & 0 & \frac{1}{8p} 
 \end{pmatrix}.
\]
\item[\rm (ii)]
The transition matrix from the basis $e^*,h^*,f^*$ to the basis $a,a^*,[a,a^*]$ is
\[
 \begin{pmatrix}
  0 & 2(1-p) & 4(1-p) \\
  1 & 1-2p & 0 \\
  0 & 2p & -4p
 \end{pmatrix},
\]
and the transition matrix from the basis $a,a^*,[a,a^*]$ to the basis $e^*,h^*,f^*$ is
\[
 \begin{pmatrix}
  \frac{2p-1}{4(1-p)}  & 1 & \frac{2p-1}{4p}   \\
  \frac{1}{4(1-p)}  & 0 & \frac{1}{4p}   \\
  \frac{1}{8(1-p)} & 0 & - \frac{1}{8p}
 \end{pmatrix}.
\]
\item[\rm (iii)]
The transition matrix from the basis $e,h,f$ to the basis $e^*,h^*,f^*$ is
\[
 \begin{pmatrix}
  p-1 & 2(1-p) & 1-p  \\
  p & 1-2p & 1-p  \\
  \frac{p^2}{1-p} & 2p & p-1
 \end{pmatrix},
\]
and the transition matrix from the basis $e^*,h^*,f^*$ to the basis $e,h,f$ is
\[
 \begin{pmatrix}
  p-1 & 2(1-p) & 1-p \\
  p & 1-2p & 1-p \\
  \frac{p^2}{1-p} & 2p & p-1
 \end{pmatrix}.
\] 
\end{itemize}
\end{lemma}

\begin{proof}
The first matrix of (i) follows from \eqref{eq:aas} and \eqref{eq:aas2}.
To get the first matrix of (ii), apply the map $*$ to \eqref{eq:aas} 
and \eqref{eq:aas2}.
The  first matrix of (iii) follows from \eqref{eq:es}--\eqref{eq:fs}.
Concerning the second matrix in (i)--(iii) just observe that it is
the inverse of the first matrix.
\end{proof}

\begin{lemma}   \label{lem:133}
For each pair of bases among \eqref{eq:threebases} the matrix
representing $\b{\;,\;}$ is given as follows:
\begin{align*}
 &
 \begin{array}{c|ccc}
   \b{\;,\;}  & e & h & f \\ \hline
   e & 0 & 0 & \frac{1}{2}  \\
   h & 0 & 1 & 0 \\
   f & \frac{1}{2} & 0 & 0
 \end{array}
 & &
 \begin{array}{c|ccc}
 \b{\;,\;}   & a & a^* & [a,a^*] \\ \hline
 a  & 1 & 1-2p & 0 \\
 a^* & 1-2p & 1 & 0 \\
 \,[a,a^*]\, & 0 & 0 & -16p(1-p)
 \end{array}
\\
\\
 &
 \begin{array}{c|ccc}
  \b{\;,\;}   & e^* & h^* & f^* \\ \hline
  e^* & 0 & 0 & \frac{1}{2} \\
  h^* & 0 & 1 & 0 \\
  f^* & \frac{1}{2} & 0 & 0
 \end{array}
 & &
 \begin{array}{c|ccc}
  \b{\;,\;}   & e^* & h^* & f^* \\ \hline
  e  & \frac{p^2}{2(1-p)}  & p & \frac{p-1}{2}  \\
  h  & p & 1-2p & 1-p \\
  f &  \frac{p-1}{2} & 1-p & \frac{1-p}{2}
 \end{array}
\\
\\
 &
 \begin{array}{c|ccc}
  \b{\;,\;}   & a & a^* & [a,a^*] \\ \hline
  e  & p & 0 & 2p \\
  h  & 1-2p & 1 & 0 \\
  f  & 1-p & 0 & 2(p-1)
 \end{array}
 & &
 \begin{array}{c|ccc}
  \b{\;,\;}   & a & a^* & [a,a^*] \\ \hline
  e^* & 0 & p & -2p \\
  h^* & 1 & 1-2p & 0 \\
  f^* & 0 & 1-p & 2(1-p)
 \end{array}
\end{align*}
\end{lemma}

\begin{proof}
The first table is from \eqref{eq:bilin}
and the second table is from Lemma \ref{lem:124}.
The remaining tables are obtained using
\eqref{eq:aas}, \eqref{eq:aas2}, \eqref{eq:es}--\eqref{eq:fs}.
\end{proof}

\begin{lemma}   \label{lem:134}
Relative to each basis \eqref{eq:threebases} the matrices 
representing $\ad a$, $\ad {a^*}$ are given as follows:
\begin{itemize}
\item[\rm (i)]
Relative to the basis $a,a^*,[a,a^*]$:
\begin{align*}
 \ad a &: 
  \begin{pmatrix}
   0 & 0 & 4(2p-1) \\
   0 & 0 & 4 \\
   0 & 1 & 0
  \end{pmatrix},
&
 \ad {a^*} &:
  \begin{pmatrix}
   0 & 0 & -4 \\
   0 & 0 & 4(1-2p) \\
   -1 & 0 & 0
  \end{pmatrix}.
\end{align*}
\item[\rm (ii)]
Relative to the basis $e,h,f$:
\begin{align*}
 \ad a &: 
  \begin{pmatrix}
   2(1-2p) & 4(p-1) & 0 \\
   -2p & 0 & 2(1-p) \\
   0 & 4p & 2(2p-1)
  \end{pmatrix},
&
 \ad {a^*} &:
  \begin{pmatrix}
   2 & 0 & 0 \\
   0 & 0 & 0 \\
   0 & 0 & -2
  \end{pmatrix}.
\end{align*}
\item[\rm (iii)]
Relative to the basis $e^*,h^*,f^*$:
\begin{align*}
 \ad a &:
  \begin{pmatrix}
   2 & 0 & 0 \\
   0 & 0 & 0 \\
   0 & 0 & -2
  \end{pmatrix},
&
 \ad {a^*} &:
  \begin{pmatrix}
   2(1-2p) & 4(p-1) & 0 \\
   -2p & 0 & 2(1-p) \\
   0 & 4p & 2(2p-1)
  \end{pmatrix}.
\end{align*}
\end{itemize}
\end{lemma}

\begin{proof}
(i):
The matrices are routinely obtained using  \eqref{eq:lem124a}
and \eqref{eq:lem124b}.

(ii), (iii):
Follows from (i) using Lemma \ref{lem:132} and the comments
above Lemma \ref{lem:134}.
\end{proof}

\medskip

By an {\em antiautomorphism} of $L$ we mean an isomorphism
of $\F$-vector spaces $\sigma : L \to L$ such that
$[y,z]^\sigma = [z^\sigma,y^\sigma]$ for $y,z \in L$.

\medskip

\begin{example}      \label{exam:anti}   \samepage
Each of the following maps is an antiautomorphism of $L$.
\begin{itemize}
\item[(i)]
The map $L \to L$, $y \mapsto -y$.
\item[(ii)]
The map $L \to L$, $y \mapsto y^t$.
\end{itemize}
\end{example}

\medskip

Consider two maps $\sigma:L\to L$ and $\tau:L\to L$, 
each of which is an automorphism or an antiautomorphism.
If exactly one is an antiautomorphism, then the composition 
$\sigma\tau$ is an antiautomorphism.
Otherwise $\sigma\tau$ is an automorphism.

\medskip

\begin{lemma}     \label{lem:anti}   \samepage
The following hold.
\begin{itemize}
\item[\rm (i)]
Let $M$ denote an invertible matrix in $\text{\rm Mat}_2(\F)$.
Then the map $L \to L$, $y \mapsto My^tM^{-1}$
is an antiautomorphism of $L$.
\item[\rm (ii)]
Let $\sigma$ denote an antiautomorphism of $L$.
Then there exists an invertible $M \in \text{\rm Mat}_2(\F)$
such that $y^\sigma = My^tM^{-1}$ for $y \in L$.
\end{itemize}
\end{lemma}

\begin{proof}
(i): Follows from Lemma \ref{lem:autoinner}(i), Example \ref{exam:anti}(ii),
and the comment below Example \ref{exam:anti}.

(ii):
The map $L \to L$, $y \mapsto (y^t)^\sigma$ is 
an automorphism of $L$.
So by Lemma \ref{lem:autoinner}(ii) there exists an invertible $M \in \text{Mat}_2(\F)$
such that $(y^t)^\sigma = MyM^{-1}$ for $y \in L$.
The result follows.
\end{proof}

\begin{lemma}     \label{lem:autibilin}   \samepage
Let $\sigma$ denote an antiautomorphism of $L$.
Then $\b{y,z} = \b{y^\sigma,z^\sigma}$ for $y,z \in L$
\end{lemma}

\begin{proof}
Define $\tau : L \to L$ such that $u^\tau = - u^\sigma$ for $u \in L$.
Then $\tau$ is an automorphism of $L$.
We have $\b{y,z} = \b{y^\tau,z^\tau}$,
so $\b{y,z} = \b{-y^\sigma,-z^\sigma} = \b{y^\sigma,z^\sigma}$.
\end{proof}

\begin{lemma}   \label{lem:135}     \samepage
There exists a unique antiautomorphism of $L$ that fixes each of
$a$, $a^*$.
Denoting this antiautomorphism by $\dagger$ we have
$(y^\dagger)^\dagger = y$ for $y \in L$.
\end{lemma}

\begin{proof}
Concerning  existence,
observe that the map $y \mapsto Wy^tW^{-1}$ is an antiautomorphism of $L$
that fixes each of $a,a^*$, where $W$ is from \eqref{eq:defUW}.
We have shown $\dagger$ exists.
We now show that $\dagger$ is unique.
Let $\dagger'$ denote an antiautomorphism of $L$ that fixes each of 
$a,a^*$. We show that $\dagger'=\dagger$.
The composition $\dagger{\dagger'}^{-1}$ is an automorphism of $L$
that fixes each of $a,a^*$, so it must be the identity map since
$a,a^*$ generate $L$. So $\dagger=\dagger'$.
Concerning the last assertion,
observe the map $y \mapsto (y^\dagger)^\dagger$ is an automorphism
of $L$ that fixes each of $a,a^*$, and hence the identity map.
\end{proof}

\begin{lemma}   \label{lem:136}     \samepage
For $y \in L$ we have $y^{\dagger} = W y^t W^{-1}$,
where $W$ is from \eqref{eq:defUW}.
\end{lemma}

\begin{proof}
The map $y \mapsto Wy^tW^{-1}$ is an antiautomorphism that fixes
each of $a,a^*$. By Lemma \ref{lem:135} such an antiautomorphism
is unique. The result follows.
\end{proof}

\begin{lemma}   \label{lem:137}     \samepage
The maps $*$ and $\dagger$ commute.
\end{lemma}

\begin{proof}
For $y=a$ and $y=a^*$ we have
$(y^*)^{\dagger} = (y^{\dagger})^*$.
\end{proof}

\begin{lemma}  \label{lem:139}     \samepage
The antiautomorphism $\dagger$ acts on $e,h,f$ and
$e^*,h^*,f^*$ in the following way:
\[
 \begin{array}{c|ccc|ccc}
  y & e & h & f & e^* & h^* & f^* \\ \hline
  y^\dagger & \frac{p}{1-p} f & h &  \frac{1-p}{p} e 
  & \frac{p}{1-p} f^* & h^* & \frac{1-p}{p}  e^*
\end{array}
\]
\end{lemma}

\begin{proof}
$e^\dagger = W e^t W^{-1} = \frac{p}{1-p}f$.
The other cases are similar.
\end{proof}

\section{Krawtchouk polynomials and the Lie algebra $\sltwo(\F)$}
\label{sec:sl2Krawt}

\indent
We continue to discuss the Lie algebra $L=\sltwo(\F)$.
In this section we consider how $L$ is related to Krawtchouk polynomials. 
We start by constructing a certain $L$-module.
Let $y,z$ denote commuting indeterminates.
Let $\F[y,z]$ denote the $\F$-algebra consisting of the polynomials
in $y,z$ that have all coefficients in $\F$.
We abbreviate ${\cal A} = \F[y,z]$.
The $\F$-vector space $\cal A$ has a basis 
\[
  y^r z^s, \qquad\qquad  r,s=0,1,2,\ldots
\]
For an integer $n \geq 0$ let $\text{Hom}_n({\cal A})$ denote the $n$th
homogeneous component of $\cal A$:
\[
  \text{Hom}_n({\cal A}) = \text{Span}\,\{y^{n-i}z^i\}_{i=0}^n.
\]
We abbreviate $H_n = \text{Hom}_n({\cal A})$.
Observe
that $\dim H_n = n+1$ and that
${\cal A}=\sum_{n=0}^\infty H_n$ (direct sum).
Moreover 
 $H_nH_m = H_{n+m}$ for $m,n \geq 0$.
We have 
 $H_0 = \F 1$
and  $H_1 = \F y + \F z$.

\medskip

For a nonzero vector space $V$ over $\F$,
let $\text{End}(V)$ denote the $\F$-algebra consisting of all $\F$-linear
transformations from $V$ to $V$.
Let $\mathfrak{gl}(V)$ denote the Lie algebra consisting of
the $\F$-vector space $\text{End}(V)$ together with Lie bracket
$[\vphi,\phi] = \vphi \phi -\phi \vphi$.

\medskip

A {\em derivation} of $\cal A$ is an element 
$\partial \in \mathfrak{gl}({\cal A})$ such that
$\partial(bc) = \partial(b)c+b \partial(c)$ for $b,c \in {\cal A}$.
Let $\text{Der}({\cal A})$ denote the set of all derivations of $\cal A$.
One checks that $\text{Der}({\cal A})$ is a Lie subalgebra of 
$\mathfrak{gl}({\cal A})$.
Observe that for $\partial \in \text{Der}({\cal A})$,
\begin{align*}
 \partial(1) &=0,  \\
 \partial(b^n) &= n b^{n-1} \partial(b),   
           & &b \in {\cal A}, \qquad\quad n=1,2,\ldots \\
 \partial(y^rz^s) &= r y^{r-1}z^s \partial(y) + s y^r z^{s-1} \partial(z),
   & & r,s=0,1,2,\ldots
\end{align*}
By these comments $\partial$ is determined by $\partial(y)$ 
and $\partial(z)$.
Therefore $\partial$ is determined by its action on
$\text{Hom}_1({\cal A})$.
We emphasize 
\begin{equation}    \label{eq:star}
 \text{$\partial = 0$ if and only if $\partial$ vanishes on $\text{Hom}_1({\cal A})$}.
\end{equation}
The following lemma asserts that any $\F$-linear transformation
$\text{Hom}_1({\cal A}) \to {\cal A}$
can be uniquely extended to $\text{Der}({\cal A})$.

\medskip

\begin{lemma}   \label{lem:140}     \samepage
For an $\F$-linear transformation $\vphi : \text{\rm Hom}_1({\cal A}) \to {\cal A}$ 
there exists a unique
$\partial=\partial_{\vphi} \in \text{\rm Der}({\cal A})$ such that
the restriction of $\partial$ on $\text{\rm Hom}_1({\cal A})$ coincides with $\vphi$.
\end{lemma}

\begin{proof}
There exists an element $\partial \in \mathfrak{gl}({\cal A})$ such that
\begin{align*}
 \partial(y^rz^s) &= r y^{r-1}z^s \vphi(y) + sy^rz^{s-1}\vphi(z),
   & & r,s=0,1,2,\ldots
\end{align*}
One checks $\partial \in \text{Der}({\cal A})$.
By construction $\partial(y)=\vphi(y)$ and $\partial(z)=\vphi(z)$,
so the restriction of $\partial$ on $\text{Hom}_1({\cal A})$ coincides with $\vphi$.
We have shown the existence of $\partial$.
The uniqueness follows from \eqref{eq:star}.
\end{proof}

\medskip

The Lie algebra $L$ acts by left multiplication on the vector
space ${\F}^{\, 2}$ (column vectors).
Recall that $\text{Hom}_1({\cal A})$ has basis $y,z$.
Consider the vector space isomorphism 
$\text{Hom}_1({\cal A}) \to {\F}^{\, 2}$ that sends
$y \mapsto (1,0)^t$ and $z \mapsto (0,1)^t$.
This isomorphism induces an $L$-module structure on 
$\text{Hom}_1({\cal A})$ such that
\begin{equation}      \label{eq:ehfaction}
 \begin{matrix}
  e.y = 0,  & \qquad & h.y= y,   &  \qquad &  f.y= z,        \\
  e.z = y, & \qquad  & \;\;\;h.z = -z, & \qquad &  f.z = 0.
 \end{matrix}
\end{equation}

\medskip

\begin{lemma}    \label{lem:141}     \samepage
The map $L \to \text{\rm Der}({\cal A})$, $\vphi \mapsto \partial_{\vphi}$
is an injective homomorphism of Lie algebras.
\end{lemma}

\begin{proof}
We first show that the map is a homomorphism of Lie algebras.
It suffices to show 
\begin{align}   \label{eq:lem141}
 \partial_{[\vphi,\phi]} &= [\partial_\vphi,\partial_\phi]
   & & (\vphi,\phi \in L).
\end{align}
In the equation \eqref{eq:lem141} both sides are contained in 
$\text{Der}({\cal A})$ and they agree on $\text{Hom}_1({\cal A})$.
So this equation holds in view of \eqref{eq:star}.
Therefore the map is a homomorphism of Lie algebras.
The injectivity is clear by construction.
\end{proof}

\medskip

We have proven the following theorem.

\medskip

\begin{theorem}    \label{thm:142}     \samepage
The algebra $\cal A$ has an $L$-module structure such that each element of $L$
acts on $\cal A$ as a derivation and \eqref{eq:ehfaction} holds.
\end{theorem}

\bigskip

For the rest of this section we fix a feasible integer $N$.
We consider the subspace $\text{Hom}_N({\cal A})$ of $\cal A$.
This subspace has a basis $\{y^{N-i}z^i\}_{i=0}^N$.

\medskip

\begin{lemma}    \label{lem:143pre}     \samepage
The elements $e,h,f$ act on the basis
$\{y^{N-i}z^i\}_{i=0}^N$ as follows: 
\begin{align*}
e.(y^{N-i}z^i) &= i y^{N-i+1}z^{i-1}  &  & (1 \leq i \leq N), 
                & e.y^N &= 0,     \\
h.(y^{N-i}z^i) &= (N-2i) y^{N-i}z^i  & & (0 \leq i \leq N),      \\
f.(y^{N-i}z^i) &= (N-i)y^{N-i-1}z^{i+1}  & & (0 \leq i \leq N-1),
                &  f.z^N &= 0.
\end{align*}
\end{lemma}
 
\begin{proof}
The element $e$ acts on $\cal A$ as a derivation,
so for $i=0,1,\ldots,N$,
\[
 e.(y^{N-i}z^i) 
  = (N-i)y^{N-i-1}z^i (e.y) + i y^{N-i}z^{i-1} (e.z).
\]
In this equation the right-hand side is equal to  $i y^{N-i+1}z^{i-1}$ 
in view of \eqref{eq:ehfaction}.
The other cases are similar.
\end{proof}

\medskip
 
The following lemma is a reformulation of Lemma \ref{lem:143pre}.
 
\medskip 

\begin{lemma}    \label{lem:143}     \samepage
With respect to the basis  $\{y^{N-i}z^i\}_{i=0}^N$ the
matrices representing $e,h,f$ are
\[
 e : \begin{pmatrix}
       0 & 1 & & & & \text{\bf 0} \\
         & 0 & 2 \\
         &   & \cdot & \cdot \\
         &   &       & \cdot & \cdot \\
         &   &       &       & 0 & N \\
       \text{\bf 0} & & &      &   & 0
     \end{pmatrix},
 \qquad\qquad
 f : \begin{pmatrix}
       0 & & & & & \text{\bf 0} \\
       N & 0 \\
         & \cdot & \cdot \\
         &     & \cdot & \cdot \\
         &     &       & 2 & 0 \\
       \text{\bf 0} & & & & 1 & 0
      \end{pmatrix},
 \]
 \[
  h : \text{\rm diag}(N,N-2,\ldots,-N).
\]
\end{lemma}

\medskip

Lemmas \ref{lem:143pre} or \ref{lem:143} shows that 
$\text{Hom}_N({\cal A})$ is an $L$-submodule of $\cal A$. 
One checks that this $L$-module is irreducible.
If $\text{Char}(\F)=0$ then
up to isomorphism $\text{Hom}_N({\cal A})$ is the unique irreducible 
$L$-module of dimension $N+1$ \cite[Theorem 7.2]{Humph}.
More generally we have the following.

\medskip

\begin{lemma}    \label{lem:unique}   \samepage
Let $V$ denote an irreducible $L$-module with dimension $N+1$.
Then the following are equivalent.
\begin{itemize}
\item[\rm (i)]
The $L$-module $V$ is isomorphic to $\text{\rm Hom}_N({\cal A})$.
\item[\rm (ii)]
$V$ has a basis $\{v_i\}_{i=0}^N$ such that 
$h.v_i = (N-2i)v_i$ for $i=0,1,\ldots,N$ and both
\[
  e.v_0 = 0, \qquad\qquad  f.v_N = 0.
\]
\end{itemize}
\end{lemma}

\begin{proof}
(i)$\Rightarrow$(ii):
Immediate from Lemma \ref{lem:143pre}.

(ii)$\Rightarrow$(i):
For $i=0,1,\ldots,N$ the vector $v_i$ is an
eigenvector for $h$ with eigenvalue $N-2i$.
Note that $\{N-2i\}_{i=0}^N$ are mutually distinct.
Pick any integer $i$ $(1 \leq i \leq N)$.
Using $[h,e]=2e$ we find $e.v_i \in \F v_{i-1}$,
and using $[h,f]=-2f$ we find $f.v_{i-1} \in \F v_i$.
Define $\alpha_i,\beta_i \in \F$ such that 
$e.v_i = \alpha_i v_{i-1}$ and $f.v_{i-1} = \beta_i v_i$.
Define $\gamma_i = \alpha_i \beta_i$.
For $i=0,1,\ldots,N$ apply each side of $[e,f]=h$ to $v_i$ and
find $\gamma_{i+1}-\gamma_i=N-2i$, where $\gamma_0=0$ and
$\gamma_{N+1}=0$.
Solving this recursion we obtain $\gamma_i=i(N-i+1)$ for $i=0,1,\ldots,N$.
Renormalizing the basis $\{v_i\}_{i=0}^N$ we may assume
$\alpha_i=i$ and $\beta_i=N-i+1$ for $1 \leq i \leq N$.
Now with respect to $\{v_i\}_{i=0}^N$ the matrices representing $e,f,h$
match those from Lemma \ref{lem:143}.
Therefore there exists an isomorphism of $L$-modules
$V \to \text{Hom}_N({\cal A})$ that sends
$v_i \mapsto y^{N-i}z^i$ for $i=0,1,\ldots,N$.
\end{proof}

\medskip

For the rest of this section we abbreviate $V=\text{Hom}_N({\cal A})$.
For $i=0,1,\ldots,N$ define $V_i = \F y^{N-i}z^i$.
We have $\dim V_i=1$ and 
\begin{equation}    \label{eq:directsum1}
   V = \sum_{i=0}^N V_i \qquad\qquad \text{ (direct sum)}.
\end{equation}
For $i=0,1,\ldots,N$ the space $V_i$ is the eigenspace of $h$ 
associated with the eigenvalue $N-2i$.
We call $V_i$ the {\em $h$-weight space} for the eigenvalue $N-2i$.
We call \eqref{eq:directsum1} the {\em $h$-weight space decomposition} of $V$.

\medskip

Recall the basis $e^*,h^*,f^*$ for $ L$ from \eqref{eq:es}--\eqref{eq:fs}.
We now describe the action of $e^*,h^*, f^*$ on the $L$-module $V$.
We will use the matrix $R$ from \eqref{eq:R}.
Recall that $y,z$ form a basis for $\text{Hom}_1({\cal A})$.
Define
\begin{align}    \label{eq:yszs}
 y^* &= (1-p)y + pz,  &  z^* &= (1-p)y + (p-1)z.
\end{align}
Then $y^*,z^*$ form a basis for $\text{Hom}_1({\cal A})$,
and $R$ is the transition matrix from $y,z$ to $y^*,z^*$.
We have
\begin{align}    \label{eq:yz}
 y &= y^* + \frac{p}{1-p}z^*, &
 z &= y^* - z^*.
\end{align}

\medskip

\begin{lemma}   \label{lem:144}     \samepage
The elements $e^*,h^*,f^*$ act on $\text{\rm Hom}_1({\cal A})$ as follows:
\begin{equation}                      \label{eq:eshsfsaction}
 \begin{matrix}
   e^*.y^* = 0, & \qquad &  h^*.y^* = y^*,  & \qquad &  f^*.y^* = z^*,
\\
 e^*.z^* = y^*,  & \qquad & \;\; h^*.z^* = -z^*, & \qquad & f^*.z^* = 0.
 \end{matrix}               
\end{equation}
\end{lemma}

\begin{proof}
Use \eqref{eq:es}--\eqref{eq:fs}, \eqref{eq:ehfaction}, and \eqref{eq:yszs}.
\end{proof}

\medskip

By construction $\{{y^*}^{N-i} {z^*}^i\}_{i=0}^N$ form a basis for $V$.

\medskip

\begin{lemma}    \label{lem:145pre}
The elements $e^*,h^*,f^*$ act on the basis $\{{y^*}^{N-i}{z^*}^i\}_{i=0}^N$
as follows:
\begin{align*}
 e^*.({y^*}^{N-i}{z^*}^i) &= i {y^*}^{N-i+1}{z^*}^{i-1} & & (1 \leq i \leq N),
 & e^*.{y^*}^N &=0,
\\
 h^*.({y^*}^{N-i}{z^*}^i) &= (N-2i) {y^*}^{N-i}{z^*}^{i} & & (0 \leq i \leq N),
\\
 f^*.({y^*}^{N-i}{z^*}^i) &= (N-i) {y^*}^{N-i-1}{z^*}^{i+1}  & & (0 \leq i \leq N-1),
 & f^*.{z^*}^N &=0.
\end{align*}
\end{lemma}

\begin{proof}
Similar to the proof of Lemma \ref{lem:143pre}
using \eqref{eq:eshsfsaction}.
\end{proof}

\medskip

The following lemma is a reformulation of Lemma \ref{lem:145pre}.

\medskip

\begin{lemma}   \label{lem:145}     \samepage
With respect to the basis $\{{y^*}^{N-i} {z^*}^i\}_{i=0}^N$
the matrices representing $e^*,h^*,f^*$ are
\[
 e^* : \begin{pmatrix}
       0 & 1 & & & & \text{\bf 0} \\
         & 0 & 2 \\
         &   & \cdot & \cdot \\
         &   &       & \cdot & \cdot \\
         &   &       &       & 0 & N \\
       \text{\bf 0} & & &      &   & 0
     \end{pmatrix},
 \qquad\qquad
 f^* : \begin{pmatrix}
       0 & & & & & \text{\bf 0} \\
       N & 0 \\
         & \cdot & \cdot \\
         &     & \cdot & \cdot \\
         &     &       & 2 & 0 \\
       \text{\bf 0} & & & & 1 & 0
      \end{pmatrix},
 \]
 \[
  h^* : \text{\rm diag}(N,N-2,\ldots,-N).
\]
\end{lemma}

\medskip

For $i=0,1,\ldots,N$ define $V^*_i = \F {y^*}^{N-i}{z^*}^i$.
We have $\dim V^*_i=1$ and 
\begin{equation}    \label{eq:directsum2}
 V = \sum_{i=0}^N V^*_i  \qquad\qquad  \text{(direct sum)}.
\end{equation}
For $i=0,1,\ldots,N$ the space $V^*_i$ is the eigenspace of $h^*$ 
associated with the eigenvalue $N-2i$.
We call $V^*_i$ the {\em $h^*$-weight space} for the eigenvalue $N-2i$.
We call \eqref{eq:directsum2} the  {\em $h^*$-weight space decomposition} of $V$.

\medskip

\begin{definition}       \label{def:ki}     \samepage
For notational convenience define
\begin{align}    \label{eq:defki}
 k_i & = \binom{N}{i} \left( \frac{p}{1-p} \right)^i
    & i &=0,1,\ldots,N.
\end{align}
Note that $k_0=1$.
\end{definition}

\medskip

We now define a bilinear form 
$\b{\;,\;} : V \times V \to \F$.
As we will see, both
\begin{align}    
 \b{V_i,V_j} &= 0   \qquad\qquad \text{if $\quad i \neq j$},
              & &  i,j=0,1,\ldots N,         \label{eq:ortho}
\\
 \b{V^*_i,V^*_j} &= 0   \qquad\qquad \text{if $\quad i \neq j$},
            & & i,j=0,1,\ldots,N.  \label{eq:orthos}
\end{align}

\medskip

\begin{definition}   \label{def:148}     \samepage
Define a bilinear form $\b{\;,\;} : V \times V \to \F$ by
\begin{align}      \label{eq:defbilin}
 \b{y^{N-i}z^i,y^{N-j}z^j}
  &= \delta_{i,j} \frac{1}{k_i (1-p)^N} ,
  & & i,j=0,1,\ldots,N
\end{align}
where $\{k_i\}_{i=0}^N$ are from Definition \ref{def:ki}.
Observe that $\b{\;,\;}$ is symmetric, nondegenerate, and
satisfies \eqref{eq:ortho}.
\end{definition}

\begin{lemma}  \label{lem:149}     \samepage
For $\vphi \in L$ and $u,v \in V$ we have
\[
  \b{\vphi.u,v} = \b{u,\vphi^\dagger.v}, 
\]
where $\dagger$ is the antiautomorphism of $L$ from Lemma {\rm \ref{lem:135}}.
\end{lemma}

\begin{proof}
Without loss of generality, we assume $\vphi$ is in the basis $e,h,f$
and  $u, v$ are in the basis $ \{y^{N-i}z^i\}_{i=0}^N$.
Write 
$u = y^{N-i}z^i$ and $v = y^{N-j}z^j$. 
First assume that $\vphi=e$.
Using Lemma \ref{lem:143pre} and \eqref{eq:defbilin},
\[
 \b{e.u,v} 
   = \b{e.(y^{N-i}z^i),y^{N-j}z^j}
   = \b{i y^{N-i+1} z^{i-1}, y^{N-j}z^j}
   = i \delta_{i-1,j} \frac{1}{k_j (1-p)^N}.
\]
By Lemma \ref{lem:139} $e^\dagger = \frac{p}{1-p}f$. 
Now using Lemma \ref{lem:143pre} and \eqref{eq:defbilin},
\[
 \b{u,e^{\dagger}.v}
  = \bBig{y^{N-i}z^i, \frac{p(N-j)}{1-p} y^{N-j-1} z^{j+1}}
  = \frac{p(N-j)}{1-p} \delta_{i-1,j} \frac{1}{k_i (1-p)^N}.
\]
By \eqref{eq:defki} we have $(1-p)ik_i = p(N-j)k_j$
provided $i-1=j$.
By these comments
$\b{e.u,v} = \b{u,e^\dagger.v}$.
The proof is similar for the case $\vphi=h$ or $\vphi=f$.
\end{proof}

\begin{lemma}   \label{lem:150}     \samepage
The bilinear form $\b{\;,\;}$ satisfies \eqref{eq:orthos}.
\end{lemma}

\begin{proof}
Let $i,j$ be given with $i \neq j$. 
Pick $u \in V^*_i$ and $v \in V^*_j$, so that
$h^*.u=(N-2i)u$ and $h^*.v = (N-2j)v$.
Observe
\[
 (N-2i)\b{u,v}
  = \b{h^*.u,v} = \b{u,(h^*)^\dagger.v}
  = \b{u,h^*.v} = (N-2j)\b{u,v}.
\]
By assumption $\text{Char}(\F) \neq 2$. 
Also since $N$ is feasible, $\text{Char}(\F)$ is $0$ or greater than $N$.
Therefore $2i \neq 2j$.
By these comments $\b{u,v} =0$.
\end{proof}

\medskip

Given a basis $\{u_i\}_{i=0}^N$ for $V$, 
there exists a unique basis $\{v_i\}_{i=0}^N$ for $V$
such that $\b{u_i,v_i} = \delta_{i,j}$ for $i,j=0,1,\ldots,N$.
The bases $\{u_i\}_{i=0}^N$ and $\{v_i\}_{i=0}^N$ are said to be
{\em dual} with respect to $\b{\;,\;}$.

\medskip

\begin{lemma}    \label{lem:151}     \samepage
With respect to $\b{\;,\;}$ the basis for $V$ dual to
$\{y^{N-i}z^i\}_{i=0}^N$ is
$\{k_i(1-p)^N y^{N-i}z^i\}_{i=0}^N$.
\end{lemma}

\begin{proof}
Immediate from \eqref{eq:defbilin}.
\end{proof}

\begin{lemma}    \label{lem:152}     \samepage
For the dual basis in Lemma {\rm \ref{lem:151}}
the sum of the basis vectors is ${y^*}^N$.
\end{lemma}

\begin{proof}
Using \eqref{eq:yszs} and \eqref{eq:defki},
\[
 \sum_{i=0}^N k_i(1-p)^N y^{N-i}z^i
 = \sum_{i=0}^N \binom{N}{i} (1-p)^{N-i} p^i y^{N-i} z^i
 = ((1-p)y+pz)^N  
 = {y^*}^N.
\]
\end{proof}

\begin{lemma}   \label{lem:153}     \samepage
For $i,j=0,1,\ldots,N$,
\begin{equation}         \label{eq:lem153}
 \b{{y^*}^{N-i}{z^*}^{i}, {y^*}^{N-j}{z^*}^j}
 = \delta_{i,j} k_i^{-1}.
\end{equation}
\end{lemma}

\begin{proof}
We assume $i=j$; otherwise \eqref{eq:lem153} holds by Lemma \ref{lem:150}.
We proceed using induction on $i$.
First assume that $i=0$.
Observe
\begin{align*}
 \left\| {y^*}^N \right\|^2
  &= \left\| \sum_{\ell=0}^N k_{\ell} (1-p)^N y^{N-\ell} z^\ell \right\|^2
              & & \text{(by Lemma \ref{lem:152})}
\\
 &= \sum_{\ell=0}^N \left\| k_{\ell} (1-p)^N y^{N-\ell}z^\ell \right\|^2
              & & \text{(by \eqref{eq:ortho})}
\\
 &= \sum_{\ell=0}^N k_\ell(1-p)^N
              & & \text{(by \eqref{eq:defbilin})}
\\
 &= \sum_{\ell=0}^N \binom{N}{\ell} (1-p)^{N-\ell} p^\ell
              & & \text{(by \eqref{eq:defki})}
\\
 &= (1-p+p)^N 
\\
 &= 1.
\end{align*}
Therefore \eqref{eq:lem153} holds for $i=0$.
Next assume that $i \geq 1$.
By Lemma \ref{lem:149},
\begin{equation}    \label{eq:lem153aux}
 \b{e^*.({y^*}^{N-i}{z^*}^i), {y^*}^{N-i+1}{z^*}^{i-1} }
 = \b{{y^*}^{N-i}{z^*}^i, (e^*)^\dagger .({y^*}^{N-i+1} {z^*}^{i-1}) }.
\end{equation}
By Lemma \ref{lem:145pre} the left-hand side of \eqref{eq:lem153aux} is equal to
$\| {y^*}^{N-i+1}{z^*}^{i-1} \|^2 i$ and this is equal to
$i k_{i-1}^{-1}$ by induction.
By Lemmas \ref{lem:139} and \ref{lem:145pre},
the right-hand side of \eqref{eq:lem153aux} is
equal to $\|{y^*}^{N-i}{z^*}^i\|^2 (N-i+1) p(1-p)^{-1}$.
By these comments 
$i k_{i-1}^{-1} = \|{y^*}^{N-i}{z^*}^i\|^2 (N-i+1) p(1-p)^{-1}$.
Evaluating this using \eqref{eq:defki} we obtain
$\|{y^*}^{N-i}{z^*}^i\|^2 = k_i^{-1}$.
Therefore \eqref{eq:lem153} holds at $i$ and the result follows.
\end{proof}

\begin{lemma}  \label{lem:154}     \samepage
With respect to $\b{\;,\;}$ the basis for $V$ dual to
$\{{y^*}^{N-i}{z^*}^i\}_{i=0}^N$ is
$\{ k_i {y^*}^{N-i}{z^*}^i\}_{i=0}^N$.
\end{lemma}

\begin{proof}
Immediate from Lemma \ref{lem:153}.
\end{proof}

\begin{lemma}   \label{lem:155}     \samepage
For the dual basis in Lemma {\rm \ref{lem:154}}
the sum of the basis vectors is $y^N$.
\end{lemma}

\begin{proof}
Using \eqref{eq:yz} and \eqref{eq:defki},
\[
\sum_{i=0}^N k_i {y^*}^{N-i}{z^*}^i
  = \sum_{i=0}^N \binom{N}{i} 
     \left( \frac{p}{1-p} \right)^i
            {y^*}^{N-i}{z^*}^i
  =  \left( y^* + \frac{p}{1-p} z^* \right)^N  
  = y^N.
\]
\end{proof}

\medskip

We have been discussing the bases 
$\{y^{N-i}z^i\}_{i=0}^N$ and $\{{y^*}^{N-i}{z^*}^i\}_{i=0}^N$ for $V$.
We now find the transition matrices between these bases.
We use Krawtchouk polynomials $\{K_i(x)\}_{i=0}^N$ from \eqref{eq:Kn}.

\medskip

\begin{lemma}   \label{lem:156}     \samepage
For $j=0,1,\ldots,N$ both
\begin{align}
 {y^*}^{N-j}{z^*}^j &=
  \sum_{i=0}^N \binom{N}{i} (1-p)^{N-i}p^i K_i(j) y^{N-i}z^i ,
                                    \label{eq:lem156a}
\\
 y^{N-j}z^j &=
  \sum_{i=0}^N \binom{N}{i}
    \left( \frac{p}{1-p} \right)^i K_i(j) {y^*}^{N-i}{z^*}^i.  \label{eq:lem156b}
\end{align}
\end{lemma}
 
\begin{proof}
We first show \eqref{eq:lem156a}.
By \eqref{eq:yz} we have $z^*=y^*-z$, so the left-hand side of
\eqref{eq:lem156a} is
\[
 {y^*}^{N-j}({y^*}-z)^j
 = {y^*}^{N-j} \sum_{\ell=0}^j \binom{j}{\ell}(-1)^\ell {y^*}^{j-\ell} z^\ell
 = \sum_{\ell=0}^N \frac{(-j)_\ell}{\ell!} {y^*}^{N-\ell} z^\ell.
\]
By \eqref{eq:Kn} the right-hand side of \eqref{eq:lem156a} is
\[
 \sum_{\ell=0}^N \frac{(-j)_\ell}{(-N)_\ell \ell! p^\ell}
  \sum_{i=0}^N \binom{N}{i} (1-p)^{N-i}p^i (-i)_\ell y^{N-i}z^i.
\]
So it suffices to show that for $\ell=0,1,\ldots,N$,
\begin{equation}    \label{eq:lem156aux1}
{y^*}^{N-\ell}z^\ell
 = \frac{1}{(-N)_\ell p^\ell}
   \sum_{i=0}^N \binom{N}{i} (1-p)^{N-i}p^i (-i)_\ell y^{N-i} z^i.
\end{equation}
In the right-hand side of \eqref{eq:lem156aux1} the $i$th term
vanishes for $i<\ell$. 
So changing the variable $r=i-\ell$, the right-hand side of
\eqref{eq:lem156aux1} becomes
\begin{align*}
\frac{1}{(-N)_\ell p^\ell}
 & \sum_{r=0}^{N-\ell} \binom{N}{r+\ell} (1-p)^{N-r-\ell}p^{r+\ell}
    (-r-\ell)_\ell y^{N-r-\ell}z^{r+\ell}
\\
 &= \frac{(-1)^\ell (N-\ell)!}{N!}
   \sum_{r=0}^{N-\ell} \frac{N!}{(r+\ell)!(N-r-\ell)!}
     (1-p)^{N-r-\ell}p^r \frac{(-1)^\ell (r+\ell)!}{r!}
        y^{N-r-\ell} z^{r+\ell}
\\
 &= z^\ell \sum_{r=0}^{N-\ell} \frac{(N-\ell)!}{r!(N-\ell-r)!}
      (1-p)^{N-\ell-r} p^r y^{N-\ell-r} z^r
\\
 &= z^\ell \sum_{r=0}^{N-\ell} \binom{N-\ell}{r}
             ((1-p)y)^{N-\ell-r} (pz)^r
\\
 &= z^\ell ((1-p)y+pz)^{N-\ell} 
\\
& = z^\ell {y^*}^{N-\ell} 
                   \qquad\qquad\qquad \text{(by \eqref{eq:yszs})}.
\end{align*}
Thus \eqref{eq:lem156aux1} holds.
We have shown \eqref{eq:lem156a}.
The proof of \eqref{eq:lem156b} is similar.
\end{proof}

\medskip

We now find the inner products between the bases
$\{y^{N-i}z^i\}_{i=0}^N$ and $\{{y^*}^{N-i}{z^*}^i\}_{i=0}^N$.

\medskip

\begin{lemma}    \label{lem:157}     \samepage
For $i,j=0,1,\ldots,N$,
\[
 \b{y^{N-i}z^i,{y^*}^{N-j}{z^*}^j} = K_i(j).
\]
\end{lemma}

\begin{proof}
Use \eqref{eq:defki}, \eqref{eq:defbilin}, and \eqref{eq:lem156a}.
\end{proof}

\medskip

Define $\F$-linear transformations $A:V \to V$ and $A^* :V \to V$ by
\begin{align}             \label{eq:defAAs}
 A &= \frac{NI-a}{2}, & A^* &= \frac{NI-a^*}{2},
\end{align}
where $a$, $a^*$ are from \eqref{eq:aas}.
Note that on $V$,
\begin{align}               \label{eq:defAAs2}
 a &= NI-2A, &  a^* &= NI-2A^*.
\end{align}

\medskip

\begin{theorem}   \label{thm:158}     \samepage
For $j=0,1,\ldots,N$ both
\begin{align}
   K_j(A)y^N &= y^{N-j} z^j,                 \label{eq:thm158a}  \\
 K_j(A^*){y^*}^N &= {y^*}^{N-j}{z^*}^j.   \label{eq:thm158b}
\end{align}
\end{theorem}

\begin{proof}
We first show \eqref{eq:thm158a}.
By \eqref{eq:hs} we have $a=h^*$.
By this and Lemma \ref{lem:145pre},
for $i=0,1,\ldots,N$ the vector ${y^*}^{N-i}{z^*}^i$ is
an eigenvector for $a$ with eigenvalue $N-2i$.
Therefore  ${y^*}^{N-i}{z^*}^i$ is
an eigenvector for $A$  with eigenvalue $i$.
Now using Lemma \ref{lem:155} along with \eqref{eq:defki}, \eqref{eq:lem156b}
we obtain
\[
K_j(A)y^N 
 = K_j(A) \sum_{i=0}^N {y^*}^{N-i}{z^*}^i k_i 
 = \sum_{i=0}^N {y^*}^{N-i}{z^*}^i k_i K_j(i)
 = y^{N-j}z^j.
\]
We have shown \eqref{eq:thm158a}.
The proof of \eqref{eq:thm158b} is similar.
\end{proof}

\medskip

For the rest of this section, we use our results so far to
easily recover some well-known properties of Krawtchouk
polynomials.

\medskip

\begin{theorem}    {\rm \cite[Section 9.11]{KLS}}   \label{thm:114}   \samepage
Krawtchouk polynomials satisfy the following orthogonality
relations:
\begin{itemize}
\item[\rm (i)]
For $i,j=0,1,\ldots,N$,
\begin{equation}                \label{eq:thm114b}
\sum_{n=0}^N K_n(i)K_n(j) \binom{N}{n} p^n(1-p)^{N-n}
 = \delta_{i,j} \binom{N}{i}^{-1} \left( \frac{1-p}{p} \right)^i.
\end{equation}
\item[\rm (ii)]
For $n,m=0,1,\ldots,N$,
\begin{equation}                      \label{eq:thm114a}
\sum_{i=0}^N K_n(i)K_m(i) \binom{N}{i} p^i(1-p)^{N-i}
 = \delta_{n,m} \binom{N}{n}^{-1}  \left( \frac{1-p}{p} \right)^n.
\end{equation}
\end{itemize}
\end{theorem}

\begin{proof}
(i):
We compute 
$\b{{y^*}^{N-i}{z^*}^i, {y^*}^{N-j}{z^*}^j}$ in two ways.
On one hand, by Lemma \ref{lem:153} and \eqref{eq:defki}
we find that $\b{{y^*}^{N-i}{z^*}^i, {y^*}^{N-j}{z^*}^j}$ is equal to
the right-hand side of \eqref{eq:thm114b}.
On the other hand, by \eqref{eq:lem156a},
\begin{align}
 {y^*}^{N-i}{z^*}^i
 &= \sum_{n=0}^N \binom{N}{n}(1-p)^{N-n}p^n K_n(i) y^{N-n}z^n,
                                                             \label{eq:thm114aux1} \\
 {y^*}^{N-j}{z^*}^j
 &= \sum_{m=0}^N \binom{N}{m}(1-p)^{N-m}p^m K_m(j) y^{N-m}z^m.
                                                             \label{eq:thm114aux2}
\end{align}
Computing $\b{{y^*}^{N-i}{z^*}^i,{y^*}^{N-j}{z^*}^j}$ using \eqref{eq:thm114aux1},
\eqref{eq:thm114aux2} and Definition \ref{def:148} we find it equals the
left-hand side of \eqref{eq:thm114b}.
Therefore \eqref{eq:thm114b} holds.

(ii):
Follows from (i) using \eqref{eq:AW}.
\end{proof}

\medskip

Krawtchouk polynomials satisfy the following three-term recurrence.

\medskip

\begin{theorem}  {\rm \cite[Section 9.11]{KLS}}    \label{thm:3termrec} \samepage
For $i,x=0,1,\ldots,N$,
\begin{equation}    \label{eq:3termrec}
 x K_i(x) = i(p-1) K_{i-1}(x) - (i(p-1)+(i-N)p) K_i(x) 
         + (i-N)p K_{i+1}(x).
\end{equation}
\end{theorem}

\begin{proof}
By Lemma \ref{lem:149} and $h^\dagger=h$,
\begin{equation}    \label{eq:3termreca}
 \b{h.(y^{N-x}z^x),{y^*}^{N-i}{z^*}^i}
   = \b{y^{N-x}z^x,h.({y^*}^{N-i}{z^*}^i)}.
\end{equation}
We first evaluate the left-hand side of \eqref{eq:3termreca}.
To do this use Lemmas  \ref{lem:143pre} and \ref{lem:157}.
We now evaluate the right-hand side of \eqref{eq:3termreca}.
By Lemma \ref{lem:132}(iii),
\begin{equation}        \label{eq:3termrecaux}
h = 2(1-p)e^* + (1-2p)h^* + 2pf^*.
\end{equation}
Evaluate the right-hand side of \eqref{eq:3termreca} using \eqref{eq:3termrecaux},
and simplify the result using Lemmas \ref{lem:145pre}, \ref{lem:157}.
By these comments \eqref{eq:3termreca} reduces to
\[
 (N-2x)K_i(x)= 2(1-p)i K_{i-1}(x) + (1-2p)(N-2i)K_i(x)+ 2p(N-i)K_{i+1}(x).
\]
In this equation we rearrange terms to get \eqref{eq:3termrec}.
\end{proof}

\medskip

Krawtchouk polynomials satisfy the following
difference equation.

\medskip

\begin{theorem} {\rm \cite[Section 9.11]{KLS}}  \label{thm:diff}   \samepage
For $i,x=0,1,\ldots,N$,
\begin{equation}        \label{eq:diff}
 i K_i(x) = x(p-1) K_i(x-1) - (x(p-1) + (x-N)p)K_i(x) + (x-N)pK_i(x+1).
\end{equation}
\end{theorem}

\begin{proof}
In \eqref{eq:3termrec} exchange $i$ and $x$,
and use \eqref{eq:AW}.
\end{proof}

\medskip

Krawtchouk polynomials have the following generating
function.

\medskip

\begin{theorem}  {\rm \cite[Section 9.11]{KLS}}    \label{thm:gen}   \samepage
Let $t$ denote an indeterminate.
Then for $x=0,1,\ldots,N$,
\begin{equation}    \label{eq:gen}
 \left( 1 - \frac{1-p}{p} t \right)^x (1+t)^{N-x}
 =  \sum_{i=0}^N \binom{N}{i} K_i(x) t^i.
\end{equation}
\end{theorem}

\begin{proof}
We apply \eqref{eq:lem156a} with $y=\frac{1}{1-p}$ and $z=\frac{t}{p}$.
Using \eqref{eq:yszs} we find $y^*=1+t$ and $z^* = 1 - \frac{1-p}{p}t$.
The result follows.
\end{proof}

\section{Description by matrices}
\label{sec:matrices}

\indent
In Section \ref{sec:sl2Krawt} we used a certain $L$-module $V$ to 
describe Krawtchouk polynomials $\{K_i(x)\}_{i=0}^N$.
In this section we summarize our results in matrix form.

\medskip
We comment on the notation. 
Recall that $\text{Mat}_{N+1}(\F)$ denotes the $\F$-algebra consisting of
all $(N+1) \times (N+1)$ matrices with entries in $\F$.
From now on,
we adopt the convention that for each matrix in this algebra 
the rows and columns are indexed by $0,1,\ldots, N$.

\medskip

\begin{definition}    \label{def:115}     \samepage
Define matrices 
$U$, $B$, $D$, $K$ in $\text{Mat}_{N+1}(\F)$ as follows.
For $i,j=0,1,\ldots,N$ the $(i,j)$-entry of $U$ is $K_i(j)$.
The matrix $B$ is tridiagonal:
\[
B = 
 \begin{pmatrix}
  a_0 & b_0 & & & & \text{\bf 0} \\
  c_1 & a_1 & b_1 \\
      & c_2 & \cdot & \cdot \\
      &     & \cdot & \cdot & \cdot \\
      &     &       & \cdot & \cdot & b_{N-1} \\
  \text{\bf 0} & &  &       & c_N   & a_N
 \end{pmatrix},
\]
where 
\begin{align}   \label{eq:aibici}
 c_i &= i(p-1), &  b_i &= (i-N)p, & a_i &= -b_i-c_i.
\end{align}
The matrix $D$ is diagonal with $(i,i)$-entry $i$ for $i=0,1,\ldots,N$.
The matrix $K$ is diagonal with $(i,i)$-entry $k_i$ for $i=0,1,\ldots,N$,
where $\{k_i\}_{i=0}^N$ are from \eqref{eq:defki}.
\end{definition}

\begin{note}
Referring to \eqref{eq:defki} and \eqref{eq:aibici},
\begin{align}   \label{eq:ki}
 k_i &= \frac{b_0b_1\cdots b_{i-1}}{c_1c_2\cdots c_i},
 & & i=0,1,\ldots,N.
\end{align}
\end{note}

\begin{theorem}   \label{thm:116}
With reference to Definition {\rm \ref{def:115}} the following hold.
\begin{itemize}
\item[\rm (i)]
$U^t=U$.
\item[\rm (ii)]
$B^t=KBK^{-1}$.
\item[\rm (iii)]
$UD=BU$.
\item[\rm (iv)]
$DU=UB^t$.
\item[\rm (v)]
$(1-p)^N UKUK = I$.
\end{itemize}
\end{theorem}

\begin{proof}
(i):
By \eqref{eq:AW}.

(ii):
One routinely checks $KB=B^tK$ by matrix multiplication,
using the tridiagonal shape of $B$ and \eqref{eq:ki}.

(iii):
This is the three-term recurrence  \eqref{eq:3termrec} in matrix form.

(iv):
This is the difference equation  \eqref{eq:diff} in matrix form.

(v):
This is the orthogonality relation  \eqref{eq:thm114a} in matrix form.
\end{proof}

\medskip

Sometimes it is convenient to work with the following matrix.

\medskip

\begin{definition}       \label{def:P}
Define $P=UK$ where the matrices $U,K$ are from 
Definition \ref{def:115}.
\end{definition}

\medskip

Theorem \ref{thm:116} looks as follows in terms of $P$.

\medskip

\begin{theorem}   \samepage
With reference to Definitions {\rm  \ref{def:115}} and {\rm \ref{def:P}}
the following hold:
\begin{itemize}
\item[\rm (i)]
 $P^t = KPK^{-1}$.
\item[\rm (ii)]
 $B^t = KBK^{-1}$.
\item[\rm (iii)]
 $PD = BP$.
\item[\rm (iv)]
 $PB = DP$.
\item[\rm (v)]
 $P^2 = (1-p)^{-N}I$.
\end{itemize}
\end{theorem}

\begin{proof}
In Theorem \ref{thm:116} eliminate $U$ using $U=PK^{-1}$.
\end{proof}        

\medskip

In Section \ref{sec:sl2Krawt} we encountered the following bases for $V$:
\begin{align}
 & \{y^{N-i}z^i\}_{i=0}^N   & & \{y^{N-i}z^i k_i(1-p)^N\}_{i=0}^N
                                                \label{eq:basis1}
\\
 & \{{y^*}^{N-i}{z^*}^i\}_{i=0}^N  & & \{{y^*}^{N-i}{z^*}^ik_i\}_{i=0}^N 
                                                 \label{eq:basis2}
\end{align}
On each line \eqref{eq:basis1}, \eqref{eq:basis2}  the two
bases on that line are dual with respect to $\b{\;,\;}$. 
We now give the transition matrices between the 
four bases in \eqref{eq:basis1}, \eqref{eq:basis2}.

\medskip

\begin{lemma}
In the diagram below we display the transition matrices between 
the four bases in \eqref{eq:basis1}, \eqref{eq:basis2}:
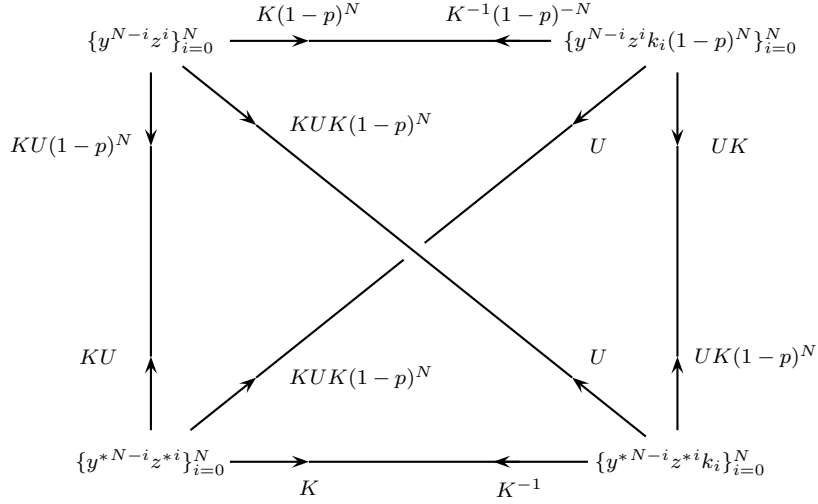
\begin{center}
\scriptsize
\begin{pspicture}(0,-2.5)(10,7)
\psset{xunit=0.7cm,yunit=0.7cm}
\rput(0,0){\rnode{A}{$\{{y^*}^{N-i}{z^*}^i\}_{i=0}^N$}}
\rput(10,0){\rnode{B}{$\;\{{y^*}^{N-i}{z^*}^ik_i\}_{i=0}^N$}}
\rput(0,8){\rnode{C}{$\{y^{N-i}z^i\}_{i=0}^N$}}
\rput(10,8){\rnode{D}{$\;\{y^{N-i}z^ik_i(1-p)^N\}_{i=0}^N$}}
\psline[arrowscale=1.5]{->}(1.5,0)(3,0)
\psline{-}(3,0)(7,0)
\psline[arrowscale=1.5]{<-}(6.5,0)(8.3,0)
\rput(3,-0.5){$K$}
\rput(7,-0.5){$K^{-1}$}
\psline[arrowscale=1.5]{->}(0,0.6)(0,2)
\psline{-}(0,2)(0,6)
\psline[arrowscale=1.5]{<-}(0,6)(0,7.4)
\rput(-1,2){$KU$}
\rput(-1.5,6){$KU(1-p)^N$}
\psline[arrowscale=1.5]{->}(1.5,8)(3,8)
\psline{-}(3,8)(7,8)
\psline[arrowscale=1.5]{<-}(6.5,8)(7.6,8)
\rput(3,8.5){$K(1-p)^N$}
\rput(7,8.5){$K^{-1}(1-p)^{-N}$}
\psline[arrowscale=1.5]{->}(10,0.6)(10,2)
\psline{-}(10,2)(10,6)
\psline[arrowscale=1.5]{<-}(10,6)(10,7.4)
\rput(11.5,2){$UK(1-p)^N$}
\rput(11,6){$UK$}
\psline[arrowscale=1.5]{->}(0.75,0.6)(2,1.6)
\psline{-}(2,1.6)(4.8,3.84)
\psline{-}(5.2,4.16)(8,6.4)
\psline[arrowscale=1.5]{<-}(8,6.4)(9.4,7.52)
\psline[arrowscale=1.5]{->}(0.6,7.52)(2,6.4)
\psline{-}(2,6.4)(8,1.6)
\psline[arrowscale=1.5]{<-}(8,1.6)(9.4,0.48)
\rput(4,1.6){$KUK(1-p)^N$}
\rput(4,6.4){$KUK(1-p)^N$}
\rput(8.5,2){$U$}
\rput(8.5,6){$U$}
\rput(5,-1.5){\rm \normalsize Transition matrices}
\rput(5,-2.5){%
$\{u_i\}_{i=0}^N \xrightarrow{\;\;\;M\;\;\;} \{v_i\}_{i=0}^N$
{\rm means} $v_j=\sum_{i=0}^N M_{ij}u_i$ $(0 \leq j \leq N)$}
\end{pspicture}
\end{center}
\end{lemma}

\begin{proof}
To get the transition matrices along the left vertical line,
reformulate  \eqref{eq:lem156a} and  \eqref{eq:lem156b}
using \eqref{eq:defki} and Definition \ref{def:115}.
The transition matrices along the two horizontal lines are
immediate from Definition \ref{def:115}.
The remaining matrices are obtained using the comments
above Lemma \ref{lem:132}.
\end{proof}
 
\medskip

We now give the inner products between the four bases in
\eqref{eq:basis1}, \eqref{eq:basis2}.

\medskip

\begin{lemma}
In the diagram below we display the inner products between 
the four bases in \eqref{eq:basis1}, \eqref{eq:basis2}:
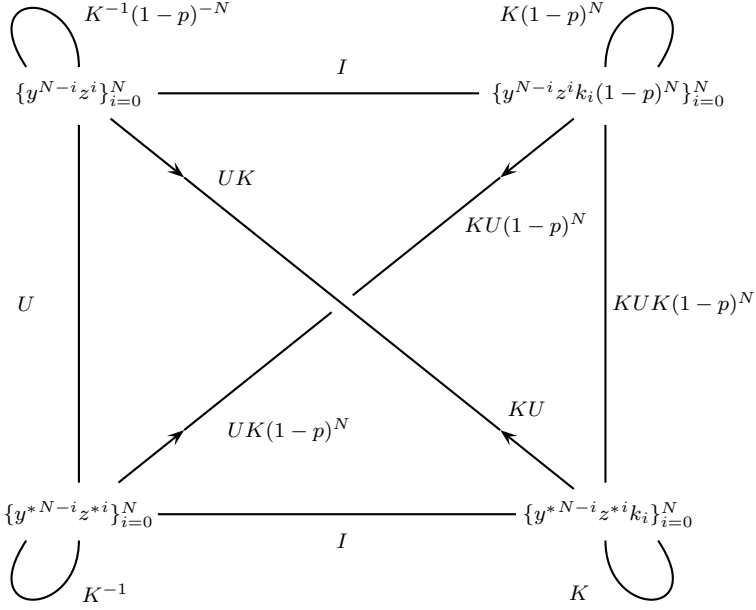
\begin{center}
\scriptsize
\begin{pspicture}(0,-4)(10,8)
\psset{xunit=0.7cm,yunit=0.7cm}
\rput(0,0){\rnode{A}{$\{{y^*}^{N-i}{z^*}^i\}_{i=0}^N$}}
\rput(10,0){\rnode{B}{$\;\{{y^*}^{N-i}{z^*}^ik_i\}_{i=0}^N$}}
\rput(0,8){\rnode{C}{$\{y^{N-i}z^i\}_{i=0}^N$}}
\rput(10,8){\rnode{D}{$\;\{y^{N-i}z^ik_i(1-p)^N\}_{i=0}^N$}}
\psline[arrowscale=1.5]{-}(1.5,0)(8.3,0)
\rput(5,-0.5){$I$}
\psline[arrowscale=1.5]{-}(0,0.6)(0,7.4)
\rput(-1,4){$U$}
\psline[arrowscale=1.5]{-}(1.5,8)(7.6,8)
\rput(5,8.5){$I$}
\psline[arrowscale=1.5]{-}(10,0.6)(10,7.4)
\rput(11.5,4){$KUK(1-p)^N$}
\psline[arrowscale=1.5]{->}(0.75,0.6)(2,1.6)
\psline{-}(2,1.6)(4.8,3.84)
\psline{-}(5.2,4.16)(8,6.4)
\psline[arrowscale=1.5]{<-}(8,6.4)(9.4,7.52)
\psline[arrowscale=1.5]{->}(0.6,7.52)(2,6.4)
\psline{-}(2,6.4)(8,1.6)
\psline[arrowscale=1.5]{<-}(8,1.6)(9.4,0.48)
\rput(4,1.6){$UK(1-p)^N$}
\rput(3,6.4){$UK$}
\rput(8.5,2){$KU$}
\rput(8.5,5.5){$KU(1-p)^N$}
\psbezier{-}(0,8.5)(0,10)(-2,10)(-1,8.5)
\psbezier{-}(10,8.5)(10,10)(12,10)(11,8.5)
\psbezier{-}(10,-0.5)(10,-2)(12,-2)(11,-0.5)
\psbezier{-}(0,-0.5)(0,-2)(-2,-2)(-1,-0.5)
\rput(1.5,9.5){$K^{-1}(1-p)^{-N}$}
\rput(9,9.5){$K(1-p)^N$}
\rput(9.5,-1.5){$K$}
\rput(0.5,-1.5){$K^{-1}$}
\rput(5,-2.5){\rm \normalsize Inner products}
\rput(5,-3.5){%
$\{u_i\}_{i=0}^N \xrightarrow{\;\;\;M\;\;\;} \{v_i\}_{i=0}^N$
{\rm means} $M_{ij}=\b{u_i,v_j}$ $(0 \leq i,j \leq N)$}
\rput(5,-4.5){%
{\rm The direction arrow is left off if} $M$ {\rm is symmetric}}
\end{pspicture}
\end{center}
\end{lemma}

\begin{proof}
Follows from Definition \ref{def:148} 
and Lemmas \ref{lem:153}, \ref{lem:157} using
Definition \ref{def:115}.
\end{proof}

\medskip

Recall the linear transformations
$A:V \to V$ and $A^*:V \to V$ from \eqref{eq:defAAs}.

\medskip

\begin{lemma}   \label{lem:147}     \samepage
With respect to our four bases in \eqref{eq:basis1}, \eqref{eq:basis2}
the matrices representing $A$ and $A^*$ are given in the table below:
\[
\begin{array}{c|cccc}
\text{\rm basis} 
 & \{y^{N-i}z^i\}_{i=0}^N
 & \{y^{N-i}z^i k_i(1-p)^N\}_{i=0}^N
 & \{{y^*}^{N-i}{z^*}^ik_i\}_{i=0}^N
 & \{{y^*}^{N-i}{z^*}^i\}_{i=0}^N   \\ \hline
A   & B^t & B & D & D \\
A^* & D & D & B & B^t
\end{array}
\]
\end{lemma}

\begin{proof}
Use \eqref{eq:aas}, \eqref{eq:defAAs} and Lemmas \ref{lem:143}, \ref{lem:145}
along with $B^t=KBK^{-1}$.
\end{proof}

\medskip

We now summarize the essential relationship between $A$ and $A^*$. 
We will use the following notion.
A tridiagonal matrix is said to be {\em irreducible} whenever
each entry on the subdiagonal is nonzero
and each entry on the superdiagonal is nonzero.
For example the tridiagonal matrix $B$ from Definition \ref{def:115} is irreducible.
Now consider the bases  $\{y^{N-i}z^i\}_{i=0}^N$ and $\{{y^*}^{N-i}{z^*}^i\}_{i=0}^N$ 
for $V$.
With respect to these bases the matrices representing $A$ and $A^*$ take the 
following form:
\[
 \begin{array}{c|c|c}
  \text{basis} & \text{matrix representing $A$} &
                 \text{matrix representing $A^*$} 
  \\ \hline
  \{y^{N-i}z^i\}_{i=0}^N & \text{irreducible tridiagonal} 
                         & \text{diagonal}  \\
  \{{y^*}^{N-i}{z^*}^i\}_{i=0}^N
    & \text{diagonal} & \text{irreducible tridiagonal}
 \end{array}
\]
In Sections \ref{sec:LP} and \ref{sec:LPKrawt} we investigate this relationship 
in a more abstract setting, 
using the notion of a Leonard pair.

\section{Leonard pairs}
\label{sec:LP}

\indent
In \cite{L} Doug Leonard characterized a family of orthogonal polynomials 
consisting of the $q$-Racah polynomials and their relatives.
This family is sometimes called the terminating branch of the 
Askey scheme \cite{KLS}, \cite[Section 24]{T:survey}.
In \cite{T:Leonard} the second author introduced the notion of a Leonard 
pair in order to clarify and simplify Leonard's characterization. 
We now define a Leonard pair.

\medskip

Throughout this section $\text{Char}(\F)$ will be arbitrary.
Let $V$ denote a vector space over $\F$ with finite positive dimension.

\medskip

\begin{definition}  \cite[Definition 1.1]{T:Leonard} \label{def:169} \samepage
By a {\em Leonard pair} on $V$ we mean an ordered pair of $\F$-linear
transformations $A:V \to V$ and $A^*:V \to V$ that satisfy
(i) and (ii) below.
\begin{itemize}
\item[(i)]
There exists a basis for $V$ with respect to which the matrix
representing $A$ is irreducible tridiagonal and the matrix
representing $A^*$ is diagonal.
\item[(ii)]
There exists a basis for $V$ with respect to which the matrix
representing $A$ is diagonal and the matrix
representing $A^*$ is irreducible tridiagonal.
\end{itemize}
By the {\em diameter} of the above Leonard pair we mean the dimension of $V$
minus one.
\end{definition}

\begin{note}   \samepage
Let $A,A^*$ denote a Leonard pair on $V$.
Then $A^*,A$ is a Leonard pair on $V$.
Also for $\alpha,\alpha^*,\beta,\beta^* \in \F$ 
with $\alpha \alpha^* \neq 0$,
the pair $\alpha A+\beta I$, $\alpha^* A^* + \beta^* I$
is a Leonard pair on $V$.
\end{note}

\medskip

The Leonard pairs are classified up to isomorphism \cite{T:Leonard,T:survey}. 
By that classification there is a natural correspondence between the 
Leonard pairs and the orthogonal polynomials that make up the terminating 
branch of the Askey scheme.
Krawtchouk polynomials are members of the terminating branch of the Askey scheme.
Our next general goal is to characterize the Leonard pairs that correspond 
to Krawtchouk polynomials.

\medskip

An element $A \in \text{End}(V)$ is said to be {\em multiplicity-free} 
whenever $A$ is diagonalizable and each eigenspace of $A$ has dimension one.

\medskip

\begin{lemma}   \label{lem:170}   \samepage
Let $A,A^*$ denote a Leonard pair on $V$.
Then each of $A,A^*$ is multiplicity-free.
\end{lemma}

\begin{proof}
Concerning $A$, by Definition \ref{def:169}(ii) there exists a basis 
for $V$ consisting of eigenvectors for $A$.
Therefore $A$ is diagonalizable.
We now show that each eigenspace of $A$ has dimension one.
To this end, we show that the number of the eigenspaces of $A$ 
is equal to the dimension of $V$. 
Note that the number of eigenspaces of $A$ is equal to the 
degree of the minimal polynomial of $A$. 
We now find this degree.
By Definition \ref{def:169}(i) there exists a basis for $V$ 
with respect to which the matrix representing $A$ 
is irreducible tridiagonal; denote this matrix by $B$. 
By construction $A$, $B$ have the same minimal polynomial. 
By the irreducible tridiagonal shape of $B$ we find that 
$I,B,B^2,\ldots,B^N$ are linearly independent, where $N= \dim V -1$. 
Therefore the minimal polynomial of $B$ has degree $N+1= \dim V$. 
By these comments the degree of the minimal polynomial of $A$ 
is equal to the dimension of $V$. 
Consequently each eigenspace of $A$ has dimension one, so
$A$ is multiplicity-free.
The case of $A^*$ is similar.
\end{proof}

\medskip

When working with a Leonard pair, it is often convenient
to consider a closely related object called a Leonard system.
In order to define this we first recall some concepts from linear algebra.
For the rest of this section set $N=\dim V-1$.
Let $A$ denote a multiplicity-free element of $\text{End}(V)$,
and let $\{\th_i\}_{i=0}^N$ denote an ordering of the eigenvalues
of $A$.
For $i=0,1,\ldots,N$ let $V_i$ denote the eigenspace of $A$
associated with $\th_i$.
So
\begin{align*}
  V &= \sum_{i=0}^N V_i   & & \text{(direct sum)}.
\end{align*}
For $i=0,1,\ldots,N$ define $E_i \in \text{End}(V)$ such that
$(E_i-I)V_i=0$ and $E_i V_j=0$ if $j \neq i$ $(j=0,1,\ldots,N)$.
We call $E_i$ the {\em primitive idempotent} of $A$ associated with $\th_i$.
Observe
(i) $E_iE_j = \delta_{i,j}E_i$ $(i,j=0,1,\ldots,N)$;
(ii) $I=\sum_{i=0}^N E_i$;
(iii) $A=\sum_{i=0}^N \th_i E_i$;
(iv) $V_i = E_i V$ $(i=0,1,\ldots,N)$.
Moreover
\begin{align}          \label{eq:Ei}
  E_i &= \prod_{\stackrel{0 \leq j \leq N}{j \neq i}}
      \frac{A-\th_j I}{\th_i-\th_j},
      & &  i=0,1,\ldots,N.
\end{align}
Let $\cal D$ denote the $\F$-subalgebra of $\text{End}(V)$ generated by $A$.
Observe that each of $\{A^i\}_{i=0}^N$ and $\{E_i\}_{i=0}^N$
is a basis for $\cal D$.
Moreover $\prod_{i=0}^N (A-\th_i I) = 0$.
Note that $\tr(E_i)=1$ and $\text{rank}(E_i)=1$
for $i=0,1,\ldots,N$.

\medskip

\begin{definition} {\rm \cite[Definition 1.4]{T:Leonard}} \label{def:171} \samepage
By a {\em Leonard system} on $V$ we mean a sequence
\[
  \Phi = (A, \{E_i\}_{i=0}^N, A^*, \{E^*_i\}_{i=0}^N)
\]
such that
\begin{itemize}
\item[(i)]
Each of $A,A^*$ is a multiplicity-free element of $\text{End}(V)$.
\item[(ii)]
$\{E_i\}_{i=0}^N$ is an ordering of the primitive idempotents of $A$.
\item[(iii)]
$\{E^*_i\}_{i=0}^N$ is an ordering of the primitive idempotents of $A^*$.
\item[(iv)]
$\displaystyle E_iA^*E_j =
  \begin{cases}
    0 & \text{ if $|i-j|>1$},    \\
    \neq 0 & \text{if $|i-j|=1$},
  \end{cases}
    \qquad\qquad i,j=0,1,\ldots,N
$.
\item[(v)]
$\displaystyle E^*_iAE^*_j =
  \begin{cases}
    0 & \text{ if $|i-j|>1$},    \\
    \neq 0 & \text{if $|i-j|=1$},
  \end{cases}
        \qquad\qquad i,j=0,1,\ldots,N
$.
\end{itemize}
\end{definition}

\medskip

Leonard pairs and Leonard systems are related as follows.
Let $A,A^*$ denote a Leonard pair on $V$.
Let $\{w_i\}_{i=0}^N$ denote a basis for $V$ from 
Definition \ref{def:169}(ii),
and let $\{w^*_i\}_{i=0}^N$ denote a basis for $V$ from 
Definition \ref{def:169}(i).
Each $w_i$ is an eigenvector for $A$; let $E_i$ denote the
corresponding primitive idempotent of $A$.
Each $w^*_i$ is an eigenvector for $A^*$; let $E^*_i$ denote the
corresponding primitive idempotent of $A^*$.
Then
$(A, \{E_i\}_{i=0}^N, A^*, \{E^*_i\}_{i=0}^N)$
is a Leonard system on $V$.
Conversely, let $\Phi=(A, \{E_i\}_{i=0}^N, A^*, \{E^*_i\}_{i=0}^N)$
denote a Leonard system on $V$.
For $i=0,1,\ldots,N$ pick nonzero vectors $w_i \in E_iV$
and $w^*_i \in E^*_iV$.
Then $\{w_i\}_{i=0}^N$ is a basis for $V$ that satisfies
Definition \ref{def:169}(ii), and
 $\{w^*_i\}_{i=0}^N$ is a basis for $V$ that satisfies
Definition \ref{def:169}(i).
Therefore $A,A^*$ is a Leonard pair on $V$.
We say the Leonard pair $A,A^*$ and the Leonard system $\Phi$
are {\em associated}.

\medskip

Let $\Phi=(A, \{E_i\}_{i=0}^N, A^*, \{E^*_i\}_{i=0}^N)$ denote a Leonard system
on $V$.
Using $\Phi$ there are several ways to get another Leonard system on $V$.
For instance, let $\alpha,\alpha^*,\beta,\beta^*$ denote scalars in $\F$ 
with $\alpha\alpha^* \neq 0$.
Then
\[
 (\alpha A+\beta I, \{E_i\}_{i=0}^N,
  \alpha^* A^* + \beta^* I, \{E^*_i\}_{i=0}^N)
\]
is a Leonard system on $V$.
Also each of the following is a Leonard system on $V$:
\begin{align*}
\Phi^* &= (A^*, \{E^*_i\}_{i=0}^N, A, \{E_i\}_{i=0}^N),
\\
\Phi^\Downarrow &= (A, \{E_{N-i}\}_{i=0}^N, A^*, \{E^*_i\}_{i=0}^N),
\\
\Phi^\downarrow &= (A, \{E_{i}\}_{i=0}^N, A^*, \{E^*_{N-i}\}_{i=0}^N).
\end{align*}

\medskip

Let $A,A^*$ denote a Leonard pair and let $\Phi$ denote an associated
Leonard system.
Then $A,A^*$ is associated with $\Phi$, $\Phi^\downarrow$,
$\Phi^{\Downarrow}$, $\Phi^{\downarrow\Downarrow}$,
and no other Leonard system.

\medskip

\begin{definition}   {\rm \cite[Definition 1.8]{T:Leonard}}
\label{def:172}   \samepage
Let $\Phi=(A, \{E_i\}_{i=0}^N, A^*, \{E^*_i\}_{i=0}^N)$ denote
a Leonard system on $V$.
For $i=0,1,\ldots,N$ let $\th_i$ (resp. $\th^*_i$) denote the
eigenvalue of $A$ (resp. $A^*$) associated with $E_i$
(resp. $E^*_i$).
We call $\{\th_i\}_{i=0}^N$ (resp. $\{\th^*_i\}_{i=0}^N$)
the {\em eigenvalue sequence}
(resp. {\em dual eigenvalue sequence}) of $\Phi$.
\end{definition}

\begin{definition} {\rm \cite[Definition 7.2]{T:LBUB}} \label{def:173} \samepage
Let $A,A^*$ denote a Leonard pair.
By an {\em eigenvalue sequence} (resp. {\em dual eigenvalue sequence})
of $A,A^*$ we mean the
eigenvalue sequence (resp. dual eigenvalue sequence) of
an associated Leonard system.
\end{definition}

\begin{note}   \samepage
Let $A,A^*$ denote a Leonard pair, with eigenvalue sequence
$\{\th_i\}_{i=0}^N$.
Then $\{\th_{N-i}\}_{i=0}^N$ is an eigenvalue sequence of $A,A^*$
and $A,A^*$ has no further eigenvalue sequence.
A similar comment applies to dual eigenvalue sequences.
\end{note}

\begin{definition}   \label{def:Phibasis}  \samepage
Let $\Phi=(A, \{E_i\}_{i=0}^N, A^*, \{E^*_i\}_{i=0}^N)$ denote
a Leonard system on $V$.
For $i=0,1,\ldots,N$ pick a nonzero $v_i \in E^*_iV$,
and note that $\{v_i\}_{i=0}^N$ is a basis for $V$.
We call such a basis a {\em $\Phi$-basis for} $V$.
\end{definition}

\begin{proposition}   \label{prop:174}   \samepage
Let  $(A, \{E_i\}_{i=0}^N, A^*, \{E^*_i\}_{i=0}^N)$ denote
a Leonard system on $V$, and let $\cal D$ denote the $\F$-subalgebra
of $\text{\rm End}(V)$ generated by $A$.
Then the $\F$-linear transformation 
${\cal D} \otimes {\cal D} \to \text{\rm End}(V)$,
$x \otimes y \mapsto x E^*_0 y$ is an isomorphism of $\F$-vector
spaces.
\end{proposition}

\begin{proof}
Recall that $\{A^i\}_{i=0}^N$ form a basis for $\cal D$, so it suffices
to show that the elements 
\begin{equation}    \label{eq:ArEs0As}
   A^r E^*_0 A^s, \qquad\qquad r,s=0,1,\ldots,N
\end{equation}
form a basis for $\text{End}(V)$.
Let $\{v_i\}_{i=0}^N$ denote a $\Phi$-basis for $V$.
Identify each element of $\text{End}(V)$ with the matrix in 
$\text{Mat}_{N+1}(\F)$ that represents it with respect to $\{v_i\}_{i=0}^N$.
From this point of view $A$ is an irreducible tridiagonal matrix and 
$E^*_0 = \text{diag}(1,0,\ldots,0)$.
For $r,s=0,1,\ldots,N$ we show that the entries of $A^rE^*_0A^s$ satisfy
\begin{align}    \label{eq:ArEs0Asaux1}
(A^rE^*_0A^s)_{ij} &=
 \begin{cases}
   0 & \text{ if $i>r$ or $j>s$},  \\
   \neq 0 & \text{ if $i=r$ and $j=s$},
 \end{cases}
  & & i,j=0,1,\ldots,N.
\end{align}
By matrix multiplication,
\begin{align}    \label{eq:ArEs0Asaux2}
 (A^rE^*_0A^s)_{ij} &= (A^r)_{i0}(A^s)_{0j},
   & & i,j=0,1,\ldots,N.
\end{align}
Using the irreducible tridiagonal shape of $A$ we find
that for $i=0,1,\ldots,N$ the entry $(A^r)_{i0}$ is zero if $i>r$,
and nonzero if $i=r$.
Similarly for $j=0,1,\ldots,N$ the entry $(A^s)_{0j}$ is zero
if $j>s$, and nonzero if $j=s$.
Combining these facts with \eqref{eq:ArEs0Asaux2} we routinely
obtain \eqref{eq:ArEs0Asaux1}.
Therefore the elements 
\eqref{eq:ArEs0As} are linearly independent.
The number of elements in \eqref{eq:ArEs0As} is $(N+1)^2$ and
this is equal to the dimension of $\text{End}(V)$.
By these comments the elements \eqref{eq:ArEs0As} form a basis
for $\text{End}(V)$.
The result follows.
\end{proof}

\begin{corollary}    \label{cor:EiEs0Ej}   \samepage
Let  $(A, \{E_i\}_{i=0}^N, A^*, \{E^*_i\}_{i=0}^N)$ denote
a Leonard system on $V$. Then the elements
\begin{equation}     \label{eq:EiEs0Ej}
   E_i E^*_0 E_j,   \qquad\qquad  i,j=0,1,\ldots,N
\end{equation}
form a basis for $\text{\rm End}(V)$.
\end{corollary}

\begin{proof}
By Proposition \ref{prop:174} and since $\{E_i\}_{i=0}^N$
form a basis for $\cal D$.
\end{proof}

\begin{lemma}    \label{lem:EsiArEsj}   \samepage
Let  $(A, \{E_i\}_{i=0}^N, A^*, \{E^*_i\}_{i=0}^N)$ denote
a Leonard system on $V$.
Then the following hold for $r=0,1,\ldots,N$.
\begin{itemize}
\item[\rm (i)]
$ \displaystyle
 E^*_iA^r E^*_j 
 = \begin{cases}
     0 & \text{ if $r < |i-j|$}, \\
     \neq 0 & \text{ if $r=|i-j|$},
   \end{cases}
    \qquad\qquad i,j = 0,1,\ldots,N$.
\item[\rm (ii)]
$ \displaystyle
 E_i{A^*}^r E_j 
 = \begin{cases}
     0 & \text{ if $r < |i-j|$}, \\
     \neq 0 & \text{ if $r=|i-j|$},
   \end{cases}
    \qquad\qquad i,j = 0,1,\ldots,N$.
\end{itemize}
\end{lemma}

\begin{proof}
(i):
Fix a $\Phi$-basis $\{v_i\}_{i=0}^N$ for $V$.
Identify each element of $\text{End}(V)$ with the matrix in
$\text{Mat}_{N+1}(\F)$ that
represents it with respect to $\{v_i\}_{i=0}^N$.
Using this point of view the result is routinely obtained.

(ii):
Similar to the proof of (i).
\end{proof}

\medskip

For the rest of this section fix a Leonard system
$\Phi=(A, \{E_i\}_{i=0}^N, A^*, \{E^*_i\}_{i=0}^N)$ on $V$, 
with eigenvalue sequence $\{\th_i\}_{i=0}^N$ and dual eigenvalue sequence 
$\{\th^*_i\}_{i=0}^N$. 
Define
\begin{align}        \label{eq:defaiasi}
  a_i &= \text{tr}(AE^*_i),  &  a^*_i &= \text{tr}(A^*E_i),
  & & i=0,1,\ldots,N.
\end{align}

\medskip

\begin{lemma}    \label{lem:EsiAEsi}   \samepage
The following hold for $i=0,1,\ldots,N$.
\begin{itemize}
\item[\rm (i)]
$E^*_iAE^*_i = a_iE^*_i$.
\item[\rm (ii)]
$E_iA^*E_i = a^*_i E_i$.
\end{itemize}
\end{lemma}

\begin{proof}
(i):
Fix a $\Phi$-basis $\{v_i\}_{i=0}^N$ for $V$.
Identify each element of $\text{End}(V)$ with the matrix in
$\text{Mat}_{N+1}(\F)$ that
represents it with respect to $\{v_i\}_{i=0}^N$.
Using this point of view the result is routinely obtained.

(ii):
Similar to the proof of (i).
\end{proof}

\begin{lemma}    \label{lem:AsAEs0}  \samepage
Assume  $N \geq 1$. Then the following hold.
\begin{align}
A^*AE^*_0 &= \th^*_1 AE^*_0 + a_0(\th^*_0-\th^*_1)E^*_0,
                                      \label{eq:AsAEs0}  \\
A^*AE^*_N &= \th^*_{N-1}AE^*_N + a_N(\th^*_N-\th^*_{N-1})E^*_N,
                                      \label{eq:AsAEsN}  \\
E_0A^*A &= \th_1 E_0A^* + a^*_0(\th_0-\th_1) E_0,
                                      \label{eq:E0AsA}  \\
E_N A^*A &= \th_{N-1}E_NA^* + a^*_N (\th_N-\th_{N-1})E_N.
                                      \label{eq:ENAsA}
\end{align}
\end{lemma}

\begin{proof}
We first show \eqref{eq:AsAEs0}.
Using $I=\sum_{i=0}^N E^*_i$ and Definition \ref{def:171}(v)
we find $AE^*_0 = I A E^*_0 = E^*_0 A E^*_0 + E^*_1 A E^*_0$.
By this and Lemma \ref{lem:EsiAEsi}(i),
\begin{equation}    \label{eq:aux1}
  AE^*_0 = a_0 E^*_0 + E^*_1 A E^*_0.
\end{equation}
In equation \eqref{eq:aux1}, multiply each side on the left by 
$A^*-\th^*_1I$,
and simplify the result using $A^* E^*_0 = \th^*_0 E^*_0$ and
$A^* E^*_1 = \th^*_1 E^*_1$.
This yields \eqref{eq:AsAEs0}.
Applying \eqref{eq:AsAEs0} to $\Phi^\downarrow$ we get \eqref{eq:AsAEsN}.
The proofs of \eqref{eq:E0AsA} and \eqref{eq:ENAsA} are similar.
\end{proof}

\begin{lemma}         \label{lem:E0AsAEs0}   \samepage
Assume $N \geq 1$. Then the following hold.
\begin{itemize}
\item[\rm (i)]
$E_0A^*AE^*_0
 = \big( (a_0-\th_0)(\th^*_0-\th^*_1) + \th_0\th^*_0 \big) E_0E^*_0$.
\item[\rm (ii)]
$E_0A^*A^2E^*_0
 = (\th_0+\th_1)E_0A^*AE^*_0 - \th_0\th_1\th^*_0 E_0E^*_0$.
\end{itemize}
\end{lemma}

\begin{proof}
(i):
In equation \eqref{eq:AsAEs0}, multiply each side on the left by $E_0$ and use
$E_0A=\th_0E_0$ to obtain the result.

(ii):
In equation \eqref{eq:E0AsA}, multiply each side on the right by 
$A-\th_0I$.
Simplify the result using $E_0A=\th_0E_0$ to find
\begin{equation}    \label{eq:aux11}
 E_0A^*A^2 = (\th_0+\th_1)E_0A^*A - \th_0\th_1 E_0A^*.
\end{equation}
In equation \eqref{eq:aux11}, multiply each side on the right by $E^*_0$,
and simplify the result using $A^*E^*_0 = \th^*_0 E^*_0$.
The result follows.
\end{proof}

\begin{lemma}      \label{lem:EsNAAsEN}   \samepage
Assume $N \geq 1$. Then the following hold.
\begin{itemize}
\item[\rm (i)]
$E^*_NAA^*E_N
 = \big( (a^*_N-\th^*_N)(\th_N-\th_{N-1}) + \th_N \th^*_N \big) E^*_NE_N$.
\item[\rm (ii)]
$E^*_NA^2A^*E_N =
  (\th_{N-1}+\th_N)E^*_NAA^*E_N - \th_{N-1}\th_N \th^*_N E^*_NE_N$.
\end{itemize}
\end{lemma}

\begin{proof}
(i):
Apply Lemma \ref{lem:E0AsAEs0}(i) to $\Phi^{*\downarrow\Downarrow}$.

(ii):
Similar to the proof of Lemma \ref{lem:E0AsAEs0}(ii).
\end{proof}

\begin{lemma}   \label{lem:a0asN}   \samepage
Assume $N \geq 1$. Then
\begin{equation}    \label{eq:a0asN}
a_0(\th^*_0 - \th^*_1) + a^*_N(\th_{N-1}-\th_N)
  = \th^*_0 \th_{N-1} - \th^*_1 \th_N.
\end{equation}
\end{lemma}

\begin{proof}
Let $\alpha$ denote the left-hand side of \eqref{eq:a0asN}
minus the right-hand side of \eqref{eq:a0asN}.
We show $\alpha=0$.
Consider the expression which is $E_N$ times  \eqref{eq:AsAEs0}
minus \eqref{eq:ENAsA} times $E^*_0$.
Simplifying this expression using
$E_NA=\th_NE_N$ and $A^*E^*_0 = \th^*_0E^*_0$ we get $\alpha E_NE^*_0=0$.
Note that $E_NE^*_0 \neq 0$ by Corollary \ref{cor:EiEs0Ej}
so $\alpha=0$.
The result follows.
\end{proof}

\section{Leonard pairs of Krawtchouk type}
\label{sec:LPKrawt}

\indent
Our discussion of $\sltwo(\F)$ in Sections \ref{sec:sl2}--\ref{sec:matrices}
was under the assumption that
$\F$ is algebraically closed with $\text{Char}(\F) \not= 2$. 
Once again we make this assumption.
For the rest of the paper fix a feasible integer $N$.

\medskip

\begin{definition}    \label{def:177}     \samepage
Let $A,A^*$ denote a Leonard pair of diameter $N$.
This Leonard pair is said to have {\em Krawtchouk type} whenever
$\{N-2i\}_{i=0}^N$ is both an eigenvalue sequence and 
dual eigenvalue sequence of $A,A^*$.
\end{definition}

\begin{lemma}   \label{lem:178}     \samepage
Let $V$ denote a vector space over $\F$ with dimension
$N+1$ and let $A,A^*$ denote a Leonard pair on $V$.
Let $\{\th_i\}_{i=0}^N$ (resp. $\{\th^*_i\}_{i=0}^N$) denote
an eigenvalue sequence (resp. dual eigenvalue sequence) of $A,A^*$.
Then the following are equivalent:
\begin{itemize}
\item[\rm (i)]
Each of $\{\th_i\}_{i=0}^N$ and $\{\th^*_i\}_{i=0}^N$ is an
arithmetic progression.
\item[\rm (ii)]
There exist scalars $\alpha,\alpha^*,\beta,\beta^*$ in $\F$
with $\alpha\alpha^* \neq 0$ such that the Leonard pair
$\alpha A + \beta I$, $\alpha^* A^* + \beta^* I$
has Krawtchouk type.
\end{itemize}
\end{lemma}

\begin{proof}
Routine.
\end{proof}

\medskip

In the following two theorems we characterize the Leonard pairs
of Krawtchouk type using $L=\sltwo(\F)$.

\medskip

\begin{theorem}   \label{thm:179}     \samepage
Consider the $L$-module $V=\text{\rm Hom}_N({\cal A})$ from 
Section {\rm \ref{sec:sl2Krawt}}.
Let $a,a^*$ denote normalized semisimple elements that generate $L$.
Then $a,a^*$ act on $V$ as a Leonard pair of Krawtchouk type.
\end{theorem}

\begin{proof}
By the comment above \eqref{eq:aas} we may assume that
the basis $e,h,f$ for $L$ is related to $a,a^*$ according to
\eqref{eq:aas}.
Consider the elements $A,A^* \in \text{End}(V)$ from \eqref{eq:defAAs}.
Recall the basis $\{y^{N-i}z^i\}_{i=0}^N$ for $V$ 
from above Lemma \ref{lem:143pre} and the basis 
$\{{y^*}^{N-i}{z^*}^i\}_{i=0}^N$ for $V$ from above Lemma \ref{lem:145pre}.
By Lemma \ref{lem:147} the action of $A,A^*$ on these bases is 
described as follows.
With respect to the basis $\{y^{N-i}z^i\}_{i=0}^N$
the matrix representing $A$ is irreducible tridiagonal 
and the matrix representing $A^*$ is $\text{diag}(0,1,\ldots,N)$.
With respect to the basis $\{{y^*}^{N-i}{z^*}^i\}_{i=0}^N$
the matrix representing $A$ is $\text{diag}(0,1,\ldots,N)$ 
and the matrix representing $A^*$ is irreducible tridiagonal.
Now by \eqref{eq:defAAs2}, with respect to the basis $\{y^{N-i}z^i\}_{i=0}^N$
the matrix representing $a$ is irreducible tridiagonal and the matrix 
representing $a^*$ is $\text{diag}(N,N-2,\ldots,-N)$.
Moreover, with respect to the basis $\{{y^*}^{N-i}{z^*}^i\}_{i=0}^N$
the matrix representing $a$ is $\text{diag}(N,N-2,\ldots,-N)$ and 
the matrix representing $a^*$ is irreducible tridiagonal.
Therefore $a,a^*$ act on $V$ as a Leonard pair of Krawtchouk type.
\end{proof}

\begin{theorem}   \label{thm:180}     \samepage
Let $V$ denote a vector space over $\F$ with dimension
$N+1$ and let $A,A^*$ denote a Leonard pair on $V$ 
that has Krawtchouk type.
Then there exists an $L$-module structure on $V$
and a pair of normalized semisimple elements of $L$ 
that generate $L$ and act on $V$ as $A,A^*$.
The $L$-module $V$ is isomorphic to the $L$-module
$\text{\rm Hom}_N({\cal A})$ from Section {\rm \ref{sec:sl2Krawt}}.
\end{theorem}

\begin{proof}
We assume $N \geq 2$; otherwise the result is routine.
Let $\Phi=(A,\{E_i\}_{i=0}^N,A^*,\{E^*_i\}_{i=0}^N)$ denote 
a Leonard system on $V$ associated with $A,A^*$.
Let $\{\th_i\}_{i=0}^N$
(resp. $\{\th^*_i\}_{i=0}^N$) denote the eigenvalue sequence
(resp. dual eigenvalue sequence) of $\Phi$.
By construction we may assume
\begin{align}   \label{eq:thithsi}
 \th_i &= N-2i,  &  \th^*_i &= N-2i, & & i=0,1,\ldots,N.
\end{align}

\medskip

{\em Claim $1$}.
{\em
$A,A^*$ satisfy both
\begin{align}
 [A,[A,[A,A^*]]] &= 4[A,A^*],     \label{eq:dg1}  \\
 [A^*,[A^*,[A^*,A]]] &= 4[A^*,A].  \label{eq:dg2}
\end{align}
}

{\em Proof.}
We first show \eqref{eq:dg1}.
Let $C$ denote the left-hand side of \eqref{eq:dg1} 
minus the right-hand side of \eqref{eq:dg1}.
Observe that
\begin{equation}    \label{eq:defC}
 C = A^3A^* - 3 A^2 A^*A + 3 A A^*A^2 - A^*A^3 - 4(AA^*-A^*A).
\end{equation}
We show $C=0$.
Since $I=\sum_{i=0}^N E_i$, it suffices to show
$E_iCE_j=0$ for $i,j=0,1,\ldots,N$.
Let $i,j$ be given.
Expand $E_iCE_j$ using \eqref{eq:defC}, and simplify
using $E_iA=\th_i E_i$ and $AE_j=\th_jE_j$ to find
\[
 E_iCE_j =  E_iA^*E_j(\th_i-\th_j+2)(\th_i-\th_j-2)(\th_i-\th_j).
\]
Observe that $E_iA^*E_j=0$ if $|i-j|>1$,
$\th_i-\th_j+2=0$ if $i-j=1$,
$\th_i-\th_j-2=0$ if $j-i=1$, and
$\th_i-\th_j=0$ if $i=j$.
In all cases $E_iCE_j=0$.
We have shown $E_iCE_j=0$ for $i,j=0,1,\ldots,N$.
Therefore $C=0$ so \eqref{eq:dg1} holds.
The proof of \eqref{eq:dg2} is similar.
We have shown Claim 1.

\medskip

{\em Claim $2$}.
{\em
There exists $p \in \F$ such that both
\begin{align}
 [A,[A,A^*]] &= 4(2p-1)A + 4A^*,   \label{eq:AW1} \\
 [A^*,[A^*,A]] &= 4(2p-1)A^* + 4A. \label{eq:AW2}
\end{align}
}

{\em Proof.}
Let $\cal D$ denote the subalgebra of $\text{End}(V)$
generated by $A$.
Since $A$ is multiplicity-free,
\[ 
  {\cal D} = \{ y \in \text{End}(V) \,|\, [y,A]=0\}.
\]
The element $[A,[A,A^*]]-4A^*$ commutes with $A$ by \eqref{eq:dg1},
so this element is contained in $\cal D$.
Therefore there exist scalars $\{\alpha_i\}_{i=0}^N$ in $\F$ such that
\begin{equation}     \label{eq:AAAsaux1}
 [A,[A,A^*]] - 4A^* = \sum_{i=0}^N \alpha_i A^i.
\end{equation}

We show $\alpha_i=0$ for $3 \leq i \leq N$.
Suppose not, and let
$k = \max\{i \,|\, 3 \leq i \leq N, \; \alpha_i \neq 0\}$.
In equation \eqref{eq:AAAsaux1}, multiply each side on the left
by $E^*_k$ and on the right by $E^*_0$.
Expand the result to find
\begin{equation}    \label{eq:claim2aux}
 E^*_k (A^2A^*-2AA^*A+A^*A^2-4A^*) E^*_0 
    = \sum_{i=0}^k \alpha_i E^*_k A^i E^*_0.
\end{equation}
Using Lemma \ref{lem:EsiArEsj}(i) we find that
the left-hand side of \eqref{eq:claim2aux} is $0$ and 
the right-hand side of \eqref{eq:claim2aux} equals
$\alpha_kE^*_kA^kE^*_0$.
Therefore $\alpha_k E^*_kA^kE^*_0 = 0$.
Recall $\alpha_k \neq 0$ by construction and
$E^*_kA^kE^*_0 \neq 0$ by Lemma \ref{lem:EsiArEsj}(i).
Therefore $\alpha_k E^*_kA^kE^*_0 \neq 0$,
for a contradiction.
We have shown $\alpha_i=0$ for $3 \leq i \leq N$.

Next we show $\alpha_2=0$.
So far we have
\[
  A^2A^*-2AA^*A+A^*A^2-4A^*
 = \alpha_0 I + \alpha_1 A + \alpha_2 A^2.
\]
In this equation, multiply each side on the left by $E^*_2$ and
on the right by $E^*_0$ to find
\begin{equation}    \label{eq:alpha2aux1}
E^*_2(A^2A^*-2AA^*A+A^*A^2-4A^*)E^*_0
 = E^*_2(\alpha_0 I + \alpha_1 A + \alpha_2 A^2)E^*_0.
\end{equation}
In \eqref{eq:alpha2aux1} we evaluate the terms in the left-hand side.
To aid in this evaluation we make some comments.
Using $A^*=\sum_{i=0}^N \th^*_iE^*_i$ and Definition \ref{def:171}(v)
we find
\begin{equation}    \label{eq:alpha2aux2}
   E^*_2AA^*AE^*_0 = \th^*_1 E^*_2AE^*_1AE^*_0.
\end{equation}
Using $I=\sum_{i=0}^N E^*_i$ and Definition \ref{def:171}(v)
we find
\begin{equation}    \label{eq:alpha2aux3}
   E^*_2A^2E^*_0 = E^*_2 A I A E^*_0
                 = E^*_2 A E^*_1 A E^*_0.
\end{equation}
Combining \eqref{eq:alpha2aux2} and \eqref{eq:alpha2aux3} we find
$E^*_2 AA^*A E^*_0 = \th^*_1 E^*_2 A^2 E^*_0$.
By these comments the left-hand side of \eqref{eq:alpha2aux1} is equal to
$(\th^*_0-2\th^*_1+\th^*_2)E^*_2A^2E^*_0$.
This is $0$ since $\th^*_0-2\th^*_1+\th^*_2=0$ by \eqref{eq:thithsi}.
The right-hand side of \eqref{eq:alpha2aux1} is equal to
$\alpha_2 E^*_2A^2E^*_0$ by Lemma \ref{lem:EsiArEsj}(i).
Therefore $\alpha_2 E^*_2A^2E^*_0 = 0$.
We have $E^*_2A^2 E^*_0 \neq 0$ by Lemma \ref{lem:EsiArEsj}(i)
so $\alpha_2=0$.

Next we show $\alpha_0 = 0$.
So far we have
\begin{equation}     \label{eq:AAAs}
  A^2A^*-2AA^*A+A^*A^2 - 4 A^* = \alpha_0 I + \alpha_1 A.
\end{equation}
In this equation we multiply each side on the left by $E_0$
and on the right by $E^*_0$.
Simplify the result using Lemma \ref{lem:E0AsAEs0}(ii) and
then Lemma \ref{lem:E0AsAEs0}(i).
Simplify the result of that using $E_0E^*_0 \neq 0$ to find
\begin{equation}     \label{eq:0}
 (a_0-\th_0)(\th_1 - \th_0)(\th^*_0-\th^*_1)
    - 4 \th^*_0 = \alpha_0 + \alpha_1 \th_0.
\end{equation}
In equation \eqref{eq:AAAs}, multiply each side on the left by $E^*_N$
and on the right by $E_N$.
Simplify the result using Lemma \ref{lem:EsNAAsEN}(ii) and
then Lemma \ref{lem:EsNAAsEN}(i).
Simplify the result of that using $E^*_NE_N \neq 0$ to find
\begin{equation}    \label{eq:N}
  (\th^*_N-a^*_N)(\th_{N-1}-\th_{N})^2 
   - 4 \th^*_N = \alpha_0 + \alpha_1 \th_N.
\end{equation}
View \eqref{eq:0}, \eqref{eq:N} as a linear system of
equations in the unknowns $\alpha_0$, $\alpha_1$.
The coefficient matrix is nonsingular since $\th_0 \neq \th_N$.
Solving this system for $\alpha_0$ and simplifying the result
using \eqref{eq:a0asN}, \eqref{eq:thithsi} we find $\alpha_0=0$.

So far we have $\alpha_i=0$ for $2 \leq i \leq N$ and
$\alpha_0=0$.
Therefore \eqref{eq:AAAsaux1} becomes
\begin{equation}
 [A,[A,A^*]]-4A^* = \alpha_1 A.    \label{eq:AW1b}
\end{equation}
Interchanging the roles of $A$ and $A^*$ in our argument so far,
we see that there exists $\alpha^*_1 \in \F$ such that
\begin{equation}
 [A^*,[A^*,A]]-4A = \alpha^*_1 A^*.  \label{eq:AW2b}
\end{equation}
We show $\alpha_1 = \alpha_1^*$.
In \eqref{eq:AW1b}, take the commutator of each term with $A^*$
to find
\begin{equation}    \label{eq:AWaux1}
 [[A,[A,A^*]],A^*] = \alpha_1 [A,A^*].
\end{equation}
Similarly using \eqref{eq:AW2b},
\begin{equation}   \label{eq:AWaux2}
 [A,[A^*,[A^*,A]]] = \alpha^*_1 [A,A^*].
\end{equation}
In \eqref{eq:AWaux1} and \eqref{eq:AWaux2}, the left-hand sides are equal by 
the Jacobi identity, so
$(\alpha_1-\alpha_1^*)[A,A^*]=0$.
Observe that $[A,A^*]\not=0$; otherwise
each of $\{E_i\}_{i=0}^N$ commutes with $A^*$ in view of \eqref{eq:Ei}, 
contradicting Definition \ref{def:171}(iv).
Therefore $\alpha_1 = \alpha^*_1$.
Now define $p \in \F$ such that $4(2p-1)=\alpha_1$.
Then \eqref{eq:AW1b}, \eqref{eq:AW2b} become
\eqref{eq:AW1}, \eqref{eq:AW2}.
We have shown Claim 2.

\medskip

{\em Claim $3$}.
{\em
We have $p \neq 0$ and $p \neq 1$.
}

\smallskip

{\em Proof}.
By Claim 2,
\begin{align}
 0 &= A^2A^* - 2AA^*A + A^*A^2 -4(2p-1)A - 4A^*, \label{eq:clm1} \\
 0 &= {A^*}^2A-2A^*AA^*+A{A^*}^2 - 4(2p-1)A^* - 4A. \label{eq:clm2}
\end{align}
In these equations, multiply each side on the left by $E^*_0$
and on the right by $E^*_0$.
Simplify the result using $E^*_0A^*=\th^*_0E^*_0$ 
and $A^*E^*_0=\th^*_0E^*_0$ to find
\begin{align}
 0 &= \th^*_0 E^*_0A^2E^*_0 - E^*_0AA^*AE^*_0 
      - 2(2p-1)E^*_0AE^*_0 -2\th^*_0E^*_0,        \label{eq:clm4}  \\
 0 &= E^*_0 A E^*_0 + (2p-1)\th^*_0 E^*_0.   \label{eq:clm3} 
\end{align}
Eliminating $E^*_0AE^*_0$ from \eqref{eq:clm4}
using \eqref{eq:clm3} we find
\begin{equation}              \label{eq:clm5}
 0 = \th^*_0 E^*_0A^2E^*_0 - E^*_0AA^*AE^*_0
        + 8 p(p-1)\th^*_0 E^*_0.
\end{equation}
Pick a $\Phi$-basis $\{v_i\}_{i=0}^N$ for $V$, and identify each
element of $\text{End}(V)$ with the matrix in $\text{Mat}_{N+1}(\F)$
that represents it with respect to $\{v_i\}_{i=0}^N$.
In \eqref{eq:clm5} we compute the $(0,0)$-entry of each side, and find
\[
   0 = (\th^*_0-\th^*_1)A_{01}A_{10} + 8 p(p-1)\th^*_0.
\]
Each of $A_{01}$, $A_{10}$ is nonzero since $A$ is irreducible
tridiagonal, so $p(p-1) \neq 0$.
We have shown Claim 3.

\medskip

Comparing Claims $2$ and $3$ with Lemma \ref{lem:127} we get
all the assertions of the theorem except the last one.
To get the last assertion we invoke Lemma \ref{lem:unique}.
Let $a,a^*$ denote a pair of normalized semisimple 
elements of $L$ that generate $L$ and act on $V$ as $A,A^*$.
By the comment above \eqref{eq:aas} we may assume
that the basis $e,h,f$ for $L$ is related to $a,a^*$ according to
\eqref{eq:aas}.
Pick a $\Phi$-basis $\{v_i\}_{i=0}^N$ for $V$.
We show that this basis satisfies condition (ii) of
Lemma \ref{lem:unique}.
By construction $h.v_i = (N-2i)v_i$ for $i=0,1,\ldots,N$.
We now show $e.v_0=0$.
We have $a^*.v_0 = N v_0$ and $a^*.v_1=(N-2)v_1$.
We also have $a.v_0 \in \text{Span}\{v_0,v_1\}$, so
there exist scalars $\xi,\eta$ in $\F$
such that $a.v_0 = \xi v_0 + \eta v_1$.
Using these comments we apply
\eqref{eq:lem124b} to $v_0$ and find $\xi=(1-2p)N$.
By Lemma \ref{lem:132}(i),
\[
  e = \frac{2a + 2(2p-1)a^* - [a,a^*]}{8(1-p)}.
\]
In this equation we apply each term to $v_0$.
Simplifying the result using the above comments we routinely find
$e.v_0=0$.
Next we show $f.v_N=0$.
We have $a^*.v_N = -Nv_N$ and $a^*.v_{N-1}=(2-N)v_{N-1}$.
We also have $a.v_N \in \text{Span}\{v_{N},v_{N-1}\}$,
so there exist scalars $\mu,\nu$ in $\F$ such that
$a.v_N = \mu v_N + \nu v_{N-1}$.
Using these comments we apply
\eqref{eq:lem124b} to $v_N$ and find $\mu=(2p-1)N$.
By Lemma \ref{lem:132}(i),
\[
 f = \frac{2a+2(2p-1)a^*+[a,a^*]}{8p}.
\]
In this equation we apply each term to $v_N$.
Simplifying the result using the above comments we routinely find
$f.v_N=0$.
We have shown that $V$ satisfies condition (ii) of Lemma \ref{lem:unique}.
By Lemma \ref{lem:unique} the $L$-module $V$ is isomorphic to the 
$L$-module $\text{Hom}_N({\cal A})$.
\end{proof}

\medskip

We have been discussing the Leonard pairs of Krawtchouk type.
For a study of general Leonard pairs we recommend the survey 
paper \cite{T:survey}.

\section{Acknowledgements}

\indent
The authors thank Erik Koelink and Tom Koornwinder for sending
us detailed historical information concerning
the Krawtchouk polynomials and $\sltwo$.

\bigskip\bigskip\noindent
Kazumasa Nomura\\
Professor Emeritus\\
Tokyo Medical and Dental University\\
Kohnodai, Ichikawa, 272-0827 Japan\\
email: knomura@pop11.odn.ne.jp

\bigskip\noindent
Paul Terwilliger\\
Department of Mathematics\\
University of Wisconsin\\
480 Lincoln Drive\\ 
Madison, Wisconsin, 53706 USA\\
email: terwilli@math.wisc.edu

\bigskip\noindent
{\bf Keywords.}
Lie algebra, Krawtchouk polynomial, Leonard pair

\noindent
{\bf 2010 Mathematics Subject Classification}.
33C05

\end{document}